\documentclass[twoside,leqno]{article}

\usepackage{amssymb}
\usepackage{amsmath}
\usepackage{amsthm}
\usepackage{ifthen}
\usepackage{calc}
\usepackage{color}
\usepackage{graphics}
\definecolor{darkgreen}{rgb}{0,0.5,0}

\def\qed{\hfill$\square$}
\numberwithin{equation}{section}
\newtheorem{thm}[equation]{\sc Theorem}
\newtheorem{lem}[equation]{\sc Lemma}
\newtheorem{cor}[equation]{\sc Corollary}
\newtheorem{prop}[equation]{\sc Proposition}

\newtheoremstyle{notation}{3pt}{3pt}{}{}{\itshape}{:}{.5em}{\thmname{#1}}
\theoremstyle{notation}

\newtheorem{defin}{\it Definition}
\newtheorem{ex}{\it Example}

\makeatletter
\renewcommand{\@seccntformat }[1]{\csname the#1\endcsname. }
\makeatother

\chardef\xnearrow='045
\chardef\xnwarrow='055
\chardef\xsearrow='046
\chardef\xswarrow='056

 1

\def\CK{{\cal K}}

\def\vsp{\vspace*{1.5ex}}

\def\ov{\overline}

\def\ind{\mbox{{\rm ind}}}
\def\id{\mbox{{\rm id}}}
\def\mod{\mbox{{\rm mod}}}

\def\Hom{\mbox{\rm Hom}}
\def\End{\mbox{\rm End}}

\def\Aut{\mbox{\rm Aut}}

\def\Ker{\mbox{\rm Ker}}

\def\Coker{\mbox{\rm Coker}}

\def\arcleq{\leq_{\rm arc}}
\def\homleq{\leq_{\rm hom}}
\def\extleq{\leq_{\rm ext}}
\def\degleq{\leq_{\rm deg}}
\def\partleq{\leq_{\rm part}}

\def\k#1{{#1}\otimes k}

\def\vsp{\vspace*{1.5ex}}

\def\vedge#1{{\buildrel{#1} \over {\hbox to
20pt{\hspace{-0.2em}$-$\hspace{-0.2em}$-$\hspace{-0.2em}$-$ }}}}

\input prepictex   \input pictex    
\input postpictex

\newcounter{boxsize}
\newcounter{tempcounter}
\newcommand{\smallentryformat}{\scriptstyle\sf}
\newcommand\smbox{\put(0,0){\line(1,0){\value{boxsize}}}%
  \put(\value{boxsize},0){\line(0,1){\value{boxsize}}}%
  \put(0,0){\line(0,1){\value{boxsize}}}%
  \put(0,\value{boxsize}){\line(1,0){\value{boxsize}}}}
\newcommand\numbox[1]{\put(0,0)\smbox%
  \put(0,0){\makebox(\value{boxsize},\value{boxsize})[c]{%
      $\smallentryformat#1$}}}
\newcommand\singlebox[1]{\raisebox{-.4ex}{\begin{picture}(4,0)\setcounter{boxsize}{3}%
    \put(0,0)\smbox%
    \put(0,0){\makebox(\value{boxsize},\value{boxsize})[c]{%
      $\scriptstyle\sf#1$}}\end{picture}}}
\newcommand\vdotbox{\setcounter{tempcounter}{\value{boxsize}*2}
  \multiput(0,-\value{boxsize})(\value{boxsize},0)2{%
    \line(0,1){\value{tempcounter}}}
  \put(0,\value{boxsize}){\line(1,0){\value{boxsize}}}
  \put(0,-2){\makebox(\value{boxsize},\value{tempcounter})[c]{%
      $\scriptscriptstyle\vdots$}}}
\def\arr#1#2{\arrow <2mm> [0.25,0.75] from #1 to #2}

\def\ssize{\scriptscriptstyle}
\newcommand\boxes[2]{\ifthenelse{#2=3}{$\scriptstyle P_2^{#1}$}{%
                                       $\scriptstyle P_{#2}^{#1}$}}

\begin{document}
\thispagestyle{empty}

%\noindent{\begin{blue}\footnotesize[{\tt \ver,} \today]\end{blue}}

\bigskip\bigskip
\begin{center}
{\large\bf Operations on Arc Diagrams and\\[1ex]
Degenerations for Invariant Subspaces of Linear Operators} 
\end{center}

\smallskip

\begin{center}
Justyna Kosakowska and Markus Schmidmeier
\footnote{The first named author is partially supported 
by Research Grant No.\ DEC-2011/02/\allowbreak A/ST1/00216
of the Polish National Science Center}\vspace{1cm}

\centerline{Dedicated to Professor Daniel Simson on the occasion of his 70th birthday}

\bigskip \parbox{10cm}{\footnotesize{\bf Abstract:}
We study geometric properties of varieties associated
with invariant subspaces of nilpotent operators. There are algebraic groups acting on these varieties.
We give dimensions of orbits of these actions.
Moreover, a~combinatorial characterization of the partial order given by degenerations  
is described. }

\medskip \parbox{10cm}{\footnotesize{\bf MSC 2010:} 
Primary: 14L30, % (group actions on varieties),   
16G20, % (representations of quivers), 
Secondary: 16G70, % (Auslander-Reiten sequences), 
05C85,  % (graph algorithms), 
47A15  % (invariant subspaces)
}

\medskip \parbox{10cm}{\footnotesize{\bf Key words:} 
degenerations, partial orders, Hall polynomials, 
nilpotent operators, invariant subspaces, Littlewood-Richardson
tableaux }
\end{center}

\section{Introduction}
%=====================

Arc diagrams represent the isomorphism types of certain invariant subspaces of
nilpotent linear operators, which in turn correspond to the orbits of the action
of an algebraic group on the representation space.
We show that operations on arc diagrams provide a combinatorial description
for the degeneration order given by this group action.

\subsection{Operations on arc diagrams}
%--------------------------------------

Two diagrams of arcs and poles are said to be in {\it arc order}
if the first is obtained from the second by a sequence of moves of type 

$$
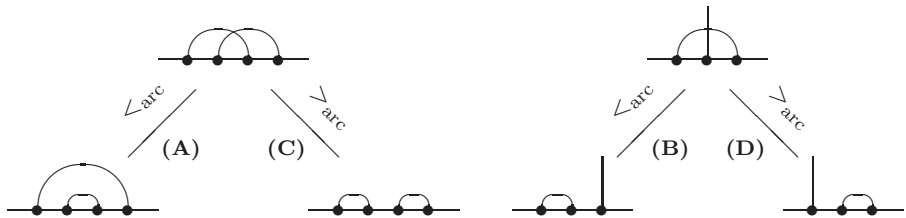
\begin{picture}(60,28)(0,-5)
  \put(0,0){\begin{picture}(20,8)(0,5)
      \put(0,4){\line(1,0){20}}
      \multiput(3,3)(4,0)4{$\bullet$}
      \put(10,4){\oval(4,4)[t]}
      \put(10,4){\oval(12,12)[t]}
    \end{picture}
  }
  \put(20,20){\begin{picture}(20,8)(0,5)
      \put(0,4){\line(1,0){20}}
      \multiput(3,3)(4,0)4{$\bullet$}
      \put(8,4){\oval(8,8)[t]}
      \put(12,4){\oval(8,8)[t]}
    \end{picture}
  }
  \put(25,15){\line(-1,-1){9}}
  \put(23,7){\makebox(0,0){\bf\footnotesize (A)}}
  \put(17,13){\rotatebox{45}{\makebox[0mm]{$<_{\rm arc}$}}}
  \put(35,15){\line(1,-1){9}}
  \put(37,7){\makebox(0,0){\bf\footnotesize(C)}}
  \put(42,12){\rotatebox{-45}{\makebox[0mm]{$>_{\rm arc}$}}}
  \put(40,0){\begin{picture}(20,8)(0,5)
      \put(0,4){\line(1,0){20}}
      \multiput(3,3)(4,0)4{$\bullet$}
      \put(6,4){\oval(4,4)[t]}
      \put(14,4){\oval(4,4)[t]}
    \end{picture}
  }
\end{picture}
\qquad
\begin{picture}(52,28)(0,-5)
  \put(0,0){\begin{picture}(16,8)(0,5)
      \put(0,4){\line(1,0){16}}
      \multiput(3,3)(4,0)3{$\bullet$}
      \put(6,4){\oval(4,4)[t]}
      \put(12,4){\line(0,1){7}}
    \end{picture}
  }
  \put(18,20){\begin{picture}(16,8)(0,5)
      \put(0,4){\line(1,0){16}}
      \multiput(3,3)(4,0)3{$\bullet$}
      \put(8,4){\oval(8,8)[t]}
      \put(8,4){\line(0,1){7}}
    \end{picture}
  }
  \put(36,0){\begin{picture}(16,8)(0,5)
      \put(0,4){\line(1,0){16}}
      \multiput(3,3)(4,0)3{$\bullet$}
      \put(10,4){\oval(4,4)[t]}
      \put(4,4){\line(0,1){7}}
    \end{picture}
  }
  \put(23,15){\line(-1,-1){9}}
  \put(21,7){\makebox(0,0){\bf\footnotesize (B)}}
  \put(15,13){\rotatebox{45}{\makebox[0mm]{$<_{\rm arc}$}}}
  \put(29,15){\line(1,-1){9}}
  \put(31,7){\makebox(0,0){\bf\footnotesize (D)}}
  \put(36,12){\rotatebox{-45}{\makebox[0mm]{$>_{\rm arc}$}}}
\end{picture}
$$

If the arc diagrams $\Delta$ and $\Delta'$ are in relation, we write $\Delta\arcleq \Delta'$.
We ask:
\begin{itemize}
\item[Q1:] For two arc diagrams, decide if they are in arc order.
\item[Q2:] If two arc diagrams are in arc order, determine a 
  sequence of moves which transforms one into the other.
\end{itemize}

Formally, an arc diagram is a finite set of arcs and poles in the 
Poincar\'e half plane.  We assume that all end points are natural numbers 
(arranged from right to left) and permit multiple arcs and poles.

\begin{ex}
The diagrams $\Delta$ and $\Delta'$ are in arc order via a single move of type {\bf (B)}.
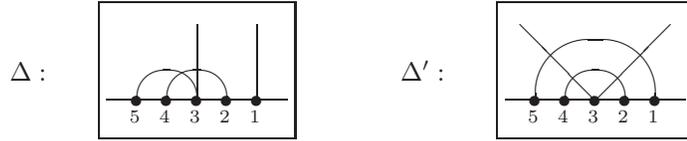
$$\raisebox{5ex}[0mm]{$\Delta:$}
  \qquad
  \framebox(26,18)[t]{\begin{picture}(24,12)(0,5)
      \put(0,4){\line(1,0){24}}
      \multiput(3,3)(4,0)5{$\bullet$}
      \put(3,1){$\scriptstyle 5$}
      \put(7,1){$\scriptstyle 4$}
      \put(11,1){$\scriptstyle 3$}
      \put(15,1){$\scriptstyle 2$}
      \put(19,1){$\scriptstyle 1$}
      \put(8,4){\oval(8,8)[t]}
      \put(12,4){\oval(8,8)[t]}
      \put(12,4){\line(0,1){10}}
      \put(20,4){\line(0,1){10}}
  \end{picture}}
  \qquad\qquad\raisebox{5ex}[0mm]{$\Delta':$}
  \qquad
  \framebox(26,18)[t]{\begin{picture}(24,12)(0,5)
      \put(0,4){\line(1,0){24}}
      \multiput(3,3)(4,0)5{$\bullet$}
      \put(3,1){$\scriptstyle 5$}
      \put(7,1){$\scriptstyle 4$}
      \put(11,1){$\scriptstyle 3$}
      \put(15,1){$\scriptstyle 2$} 
      \put(19,1){$\scriptstyle 1$}
      \put(12,4){\oval(16,16)[t]}
      \put(12,4){\oval(8,8)[t]}
      \put(12,4){\line(1,1){10}}
      \put(12,4){\line(-1,1){10}}
  \end{picture}}
  $$
\end{ex}

\subsection{Short exact sequences of linear operators}
%-----------------------------------------------------

Let $k$ be a field.  We call a $k[T]$-module
$N_\alpha=N_\alpha(k)=\bigoplus_{i=1}^s k[T]/(T^{\alpha_i})$, where $\alpha=(\alpha_1\geq\cdots\geq\alpha_s)$
is a partition, a nilpotent linear operator.
A monomorphism between two nilpotent linear operators is an {\it embedding} of
an {\it invariant subspace.} 
% \begin{red}Suppose now that the field $k$ is algebraically closed.
% \end{red}
Given three partitions $\alpha,\beta,\gamma$, we consider the 
%\begin{red}constructive\end{red} 
subset $V_{\alpha,\gamma}^\beta(k)$ of $H_\alpha^\beta(k)=\Hom_k(N_\alpha,N_\beta)$
of all embeddings $f:N_\alpha\to N_\beta$ 
of $k[T]$-modules which give rise to a short exact sequence
$$0\longrightarrow N_\alpha\stackrel f\longrightarrow N_\beta
   \longrightarrow N_\gamma\longrightarrow 0.$$
%\begin{green} 
Suppose now that the field $k$ is algebraically closed.
Then the subset $V_{\alpha,\gamma}^\beta(k)\subset H_\alpha^\beta(k)$ 
is constructible and there is an algebraic group acting on it
%\end{green}
such that the isomorphism classes
of short exact sequences are in one-to-one correspondence 
with the orbits under this group action.

\medskip 
For points $Y=(N_\alpha,N_\beta,f)$, $Z=(N_\alpha,N_\beta,g)$ in $V_{\alpha,\gamma}^\beta$,
we define $Y\degleq Z$ if $g$ occurs in the closure of the orbit $G_f$ of $f$ under this group 
action.  We ask:

\begin{itemize}
\item[Q3:] Given two embeddings $Y,Z\in V_{\alpha,\gamma}^\beta$, does the relation $Y\degleq Z$ hold?
\item[Q4:] If $Y\degleq Z$, can we find a sequence $Y=Y_0\degleq Y_1\degleq\cdots\degleq Y_s=Z$
  such that the dimensions of two subsequent orbits differ by one?
\end{itemize}

It is the aim of this manuscript to shed light on how Q1 and Q3 are related, 
and to provide an algorithm for Q2 which in turn yields a critereon for Q4.

\subsection{From short exact sequences to arc diagrams}\label{section-from-ses}
%------------------------------------------------------

In case the partition $\alpha$ has first (or equivalently all) entries at most 2,
the isomorphism types of embeddings in $V_{\alpha,\gamma}^\beta$ are determined combinatorially:

\begin{prop}[\protect{\cite[Proposition~2]{sch}}]
  Let $k$ be any field.
  For partitions $\alpha,\beta,\gamma$ with $\alpha_1\leq 2$, there is a one-to-one correspondence
  $$\big\{\text{embeddings in $V_{\alpha,\gamma}^\beta(k)$}\big\}{\big/}_{\cong} 
  \quad\stackrel{1-1}\longleftrightarrow 
  \quad \big\{\text{Klein tableaux of type $(\alpha,\beta,\gamma)$}\big\}.$$
\end{prop}

A Klein tableau of type $(\alpha,\beta,\gamma)$ is a skew diagram of shape $\beta\backslash\gamma$
with $\bar\alpha_1$ symbols $\singlebox1$ and $\bar\alpha_2$ entries $\singlebox{2_r}$ 
for suitable subscripts $r$, see Section~\ref{section-pickets}. 
Here $\bar\alpha$ denotes the conjugate of $\alpha$,
so the condition $\alpha_1\leq2$
implies that all entries in the tableau are at most 2.

\smallskip
A Klein tableau $\Pi$ determines  an {\it arc diagram,} as follows.
Suppose $\Pi$ has $e=\beta_1$ rows.  
\begin{itemize} 
\item Arrange the vertices $e, e-1,\ldots, 1$ on a horizontal line.  
\item For each symbol $\singlebox{2_r}$ in row $m$,  draw
  an arc above the line connecting $m$ with $r$.
\item If the number of arcs ending at $r$ is less than the 
  number of symbols $\singlebox{1}$ in row $r$,
  draw for each remaining symbol a vertical line above $r$.
\end{itemize}

\begin{ex}
The Klein tableau $\Pi$ has the arc diagram $\Delta$.

\vspace{-4mm}
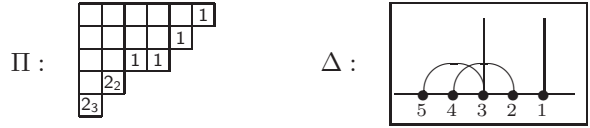
$$\Pi:\quad
\begin{picture}(18,12)(0,6)
\multiput(0,12)(3,0)5{\smbox}
\put(15,12){\numbox{1}}
\multiput(0,9)(3,0)4{\smbox}
\put(12,9){\numbox{1}}
\multiput(0,6)(3,0)2{\smbox}
\multiput(6,6)(3,0)2{\numbox{1}}
\put(0,3){\smbox}
\put(3,3){\numbox{2_2}}
\put(0,0){\numbox{2_3}}
\end{picture}
\qquad\qquad
\Delta:\quad
\raisebox{-7mm}{\framebox(26,16)[t]{\begin{picture}(24,12)(0,4)
    \put(0,4){\line(1,0){24}}
    \multiput(3,3)(4,0)5{$\bullet$}
    \put(3,1){\makebox{$\scriptstyle 5$}}
    \put(7,1){\makebox{$\scriptstyle 4$}}
    \put(11,1){\makebox{$\scriptstyle 3$}}
    \put(15,1){\makebox{$\scriptstyle 2$}}
    \put(19,1){\makebox{$\scriptstyle 1$}}
    \put(8,4){\oval(8,8)[t]}
    \put(12,4){\oval(8,8)[t]}
    \put(12,4){\line(0,1){10}}
    \put(20,4){\line(0,1){10}}
  \end{picture}
}}
$$
\end{ex}

\begin{defin}
The {\it arc diagram} $\Delta(Y)$ of an invariant subspace 
$Y\in V_{\alpha,\gamma}^\beta$ is defined to be the
arc diagram given by the Klein tableau 
representing the isomorphism class of $Y$ in $V_{\alpha,\gamma}^\beta$.
We say two invariant subspaces $Y,Z\in V_{\alpha,\gamma}^\beta$ 
are in {\it arc order,} in symbols $Y\arcleq Z$, if $\Delta(Y)\arcleq \Delta(Z)$
holds.
\end{defin}

\subsection{Main results}
%------------------------

We can now relate the above problems Q1 and Q3:

\begin{thm}\label{thm-first-main}
Suppose that $k$ is an algebraically closed field and that $\alpha,\beta,\gamma$
are partitions with $\alpha_1\leq 2$. 
For invariant subspaces $Y,Z\in V_{\alpha,\gamma}^\beta$ we have
$$Y\degleq Z\qquad\text{if and only if}\qquad Y\arcleq Z.$$
\end{thm}

In the proof of the ``only if'' part, we will present an algorithm which
delivers a sequence of arc moves that convert $Z$ to $Y$, addressing Q2.

\medskip
The diagram of an invariant subspace determines the dimension of its orbit 
in $V_{\alpha,\gamma}^\beta$ as follows.  
% For a partition $\lambda$, the {\it moment}
% is given by $n(\lambda)=\sum_{i}\lambda_i(i-1)$. 
% 
% \begin{thm}\label{thm-second-main}
% Let $k$ be an algebraically closed field, and let $\alpha,\beta,\gamma$ be partitions.
% Suppose the arc diagram $\Delta$ of an invariant subspace 
% $Y=(N_\alpha,N_\beta,f)\in V_{\alpha,\gamma}^\beta(k)$ has $x(\Delta)$ crossings. Then
% $$\dim G_f\;=\; n(\beta)-n(\alpha)-n(\gamma)-x(\Delta)\blue{+|\alpha|+2n(\alpha)},$$
% \blue{where $|\alpha|=\alpha_1+\alpha_2+\ldots$}
%   \end{thm}
% 
%\begin{green}
For a partition $\lambda$, the {\it length} is given by 
$|\lambda|=\lambda_1+\lambda_2+\cdots$, and the {\it moment}
is $n(\lambda)=\sum_{i}\lambda_i(i-1)$. 
\begin{thm}\label{thm-second-main}
Let $k$ be an algebraically closed field, and let $\alpha,\beta,\gamma$ be partitions.
Suppose the arc diagram $\Delta$ of an invariant subspace 
$Y=(N_\alpha,N_\beta,f)\in V_{\alpha,\gamma}^\beta(k)$ has $x(\Delta)$ crossings. Then
$$\dim G_f\;=\; \deg h_{\alpha,\gamma}^\beta + \deg a_\alpha -x(\Delta),$$
where $\deg h_{\alpha,\gamma}^\beta = n(\beta)-n(\alpha)-n(\gamma)$ 
is the degree of the Hall polynomial $h_{\alpha,\gamma}^\beta(q)$ and
$\deg a_\alpha=|\alpha|+2 n(\alpha)$ is the degree of the polynomial 
$a_\alpha(q)$ which counts the automorphisms of $N_\alpha(\mathbb F_q)$. 
\end{thm}
%\end{green}

As a consequence we obtain a critereon for Q4:  
Assume $Z$ can be converted to $Y$ via a sequence of arc moves such 
that each move reduces the number of crossings by 1.  Then there exists a sequence
of orbits in $V_{\alpha,\gamma}^\beta$ such that in each step the dimension increases by 1,
and conversely.  The first example in Section~\ref{subsec-lattice} shows that 
in general the answer to Q4 is No (consider the Klein tableaux $\Pi_7$ and $\Pi_5$,
or $\Pi_6$ and $\Pi_4$).

\subsection{Diagrams of oriented arcs}
%-------------------------------------

Note that the arc moves {\bf (A)} and {\bf (B)} share the property that at each vertex,
the number of incoming arcs and the number of outgoing arcs remains constant,
where poles are considered as incoming.  
It turns out that the question whether two arc diagrams are in arc order 
via moves of type {\bf (A)} or {\bf (B)} has an interpretation in terms of 
linear operators.

\medskip
Recall the Theorem by Green and Klein, which we present in the version
for linear operators.

\begin{thm}[\protect{\cite{gk}}]
Let $k$ be any field, and let $\alpha,\beta,\gamma$ be partitions.  
There exists a short exact sequence of linear operators
$$0\longrightarrow N_\alpha\longrightarrow N_\beta\longrightarrow N_\gamma\longrightarrow 0$$
if and only if there exists a Littlewood-Richardson (LR-)tableau of type $(\alpha,\beta,\gamma)$.
\end{thm}

We will consider LR-tableaux in Section~\ref{section-lr-tableau}; for the purpose
of this paragraph, an LR-tableau $\Gamma$ is just a Klein tableau $\Pi$ 
with all subscripts omitted.
We call $\Gamma$ the {\it underlying} LR-tableau of $\Pi$, 
and $\Pi$ a {\it refinement} of $\Gamma$.
Suppose an arc diagram $\Delta$ is given by a Klein tableaux $\Pi$.
Observe that exactly the arc moves of type {\bf (A)} and {\bf (B)}
are given by changing subscripts in the Klein tableau, and hence
leave the underlying LR-tableau unchanged.

\medskip
Suppose $\Gamma$ is an LR-tableau of type $(\alpha,\beta,\gamma)$.
Denote by $V_\Gamma$ the constructible subset
of $H_\alpha^\beta$ (contained in 
$V_{\alpha,\gamma}^\beta$) of all invariant subspaces 
$Y$ for which the Klein tableau representing the isomorphism type of $Y$ 
is a refinement of $\Gamma$. 

\begin{thm}\label{thm-third-main}
Suppose $k$ is an algebraically closed field, and $\Gamma$ is an LR-tableau
with entries at most $2$.
For invariant subspaces $Y,Z\in V_\Gamma$ we have
$Y\degleq Z$ if and only if $Y\arcleq Z$.
If one of the relations holds, then $Y$ is obtained from $Z$ by a sequence of
arc moves of type {\bf (A)} and {\bf (B)}.
\end{thm}

\subsection{History and related results}
%---------------------------------------

On the sets of orbits in $V_{\alpha,\gamma}^\beta(k)$ we
consider the partial order $\leq_{\rm deg}$ given by degenerations and the 
partial orders $\leq_{\rm ext}$, $\leq_{\rm hom}$ (see Section \ref{sec-partial-orders}).
It is well known that 
$$\leq_{\rm ext}  \quad\Longrightarrow\quad
\leq_{\rm deg}\quad \Longrightarrow\quad \leq_{\rm hom} $$ 
(see \cite{bongartz}, \cite{riedtmann}). 
But the implication $\leq_{\rm hom}\, \Longrightarrow\,
\leq_{\rm ext} $ is not always true. There is an~open problem to
find classes of algebras or modules for which the last implication
holds. This is the case for representations of Dynkin and extended
Dynkin quivers (\cite{bongartz}, \cite{bongartz1}, \cite{zwara}),
 for representation directed algebras (\cite{bongartz}),
 for 
posets of finite prinjective type (\cite{kos03}).
For more comprehensive information about degenerations of modules we refer the
reader  to \cite{bongartz}, \cite{bongartz1},
\cite{riedtmann}.

\medskip
The problem of classifying invariant subspaces, up to isomorphism, 
is studied since Birkhoff \cite{birkhoff}, where he challenges us to
classify all embeddings of a subgroup in an
abelian group, up to automorphisms of the ambient group.
There are many variations of the original problem, as one can take for the coefficient ring
any local uniserial ring, and as one can admit several submodules. 
For a detailed investigation of the representation theoretic complexity of many categories
of poset representations with coefficients in a local uniserial ring 
we refer the reader to \cite{simson}.

\smallskip
We are here interested in the category $\mathcal S_2(k)$ of all embeddings
$f:N_\alpha\to N_\beta$ where $\alpha,\beta$ are partitions such that $\alpha_1\leq 2$.  
This is a full subcategory, closed under subobjects, of the category
of representations of the one point poset with coefficients in the local uniserial ring $k[[T]]$.
The category $\mathcal S_2(k)$ is of discrete representation type, but not representation directed.

\subsection{Organization of this paper}
%--------------------------------------
In Section \ref{sec-definitions} 
we give definitions and notation concerning LR-tableaux, Klein tableaux,
arc diagrams and the category $\mathcal{S}_2(k)$.
In the proof of Lemma~\ref{lemma-crossings} we show how a formula by T.~Klein in~\cite{klein}
computes the number $x(\Delta)$ of crossings in an arc diagram $\Delta$
as the deviation $x(\Pi)$ from dominance of the underlying Klein tableau.

\smallskip
The partial orders $\leq_{\rm arc}$, $\leq_{\rm ext}$,
  $\leq_{\rm deg}$ and $\leq_{\rm hom}$ are introduced in Section \ref{sec-partial-orders};
their relation is discussed in Theorem~\ref{thm-main} which also deals with the
case where the base field is not algebraically closed.
We show the following implications 
$$
    \leq_{\rm ext} \quad\Longrightarrow \quad
    \leq_{\rm deg} \quad\Longrightarrow \quad\leq_{\rm hom}. $$  

\smallskip
In Section \ref{chapter-hom-tab} we complete the proofs 
of Theorems~\ref{thm-first-main} and \ref{thm-third-main}
by showing the implications
   $$ \leq_{\rm hom} \quad\Longrightarrow \quad\leq_{\rm arc}
      \quad\Longrightarrow\quad\extleq.$$
For the first, we present an algorithm which determines for given
objects $Y,Z\in\mathcal S_2$ satisfying $Y\homleq Z$
a sequence of arc operations that transform the arc diagram for $Z$ into
the arc diagram for $Y$.

\smallskip
%\begin{green}
Let $\Pi$ be a Klein tableau, $q$ a prime power and $k$ the algebraic closure
of $\mathbb F_q$.  Denote by $g_\Pi(q)$ the number 
of monomorphisms in $V_{\alpha,\gamma}^\beta(\mathbb F_q)$ corresponding to $\Pi$,
and let $Y=(N_\alpha,N_\beta,f)$ be a monomorphism in $V_{\alpha,\gamma}^\beta(k)$,
also corresponding to $\Pi$.
%\end{green}
In Section~\ref{sec-dim-hall} we use a result by Lang and Weil \cite{lw}
which links the dimension of the variety $G_f(k)$ to the degree of the polynomial
$g_\Pi(q)$, 
and hence to the degree of the Hall polynomial $h_{\alpha,\gamma}^\beta(q)$, the cardinality
$a_\alpha(q)$ of the automorphism group of $N_\alpha(\mathbb F_q)$ and the 
number of crossings $x(\Delta)$ of the arc diagram given by $\Pi$. 

In order to apply this result, we verify in Lemma~\ref{lemma-defined-over-Z} that
the category $\mathcal S_2$ can be defined over $\mathbb Z$.
This finishes the proof of Theorem~\ref{thm-second-main}.
Finally, in Theorem~\ref{thm-lattice} we describe the minimal and the maximal
elements in the partially ordered set of all arc diagrams corresponding to a
given partition type.

\section{Notation and definitions}\label{sec-definitions}
%=================================

We introduce Klein tableaux, LR-tableaux and arc diagrams as they provide 
isomorphism invariants for the objects in the category $\mathcal S_2$.

\subsection{From tableaux to arc diagrams}\label{section-lr-tableau}
%-----------------------------------------

For a partition $\alpha$ we denote by  $\bar\alpha$ the conjugate, 
so  $\alpha_i$ is the length of the $i$-th column and $\bar\alpha_j$ the length of the 
$j$-th row of the corresponding Young diagram.  
The sum of the entries of $\alpha$ is denoted by $|\alpha|$. 

\begin{defin}
  \begin{enumerate}
  \item 
    Given three partitions $\alpha,\beta,\gamma$, 
    an {\it LR-tableau} of type $(\alpha,\beta,\gamma)$ is a skew diagram 
    of shape $\beta\backslash \gamma$ 
    with $\bar\alpha_1$ entries $\singlebox 1$,
    $\bar\alpha_2$ entries $\singlebox 2$, etc.
    The entries are weakly increasing in each row, strictly
    increasing in each column, and satisfy the lattice permutation property 
    (for each $c\geq 0$, $\ell\geq 2$ there are at least 
    as many entries $\ell-1$ on the 
    right hand side of the $c$-th column as there are entries $\ell$).
  \item
    A {\it Klein tableau} 
    of type $(\alpha,\beta,\gamma)$
    is an LR-tableau of the same type
      where in addition each entry $\ell\geq 2$ carries a subscript,
        subject to the conditions (see \cite[(1.2)]{klein}):
        \begin{enumerate}
        \item If a symbol $\singlebox {\ell_r}$ occurs in the $m$-th row in the tableau, 
          then $1\leq r\leq m-1$.
        \item If $\singlebox{\ell_r}$ occurs in the $m$-th row
          and the entry above $\singlebox{\ell_r}$ is $\ell-1$, 
          then $r=m-1$.
        \item The total number of symbols $\singlebox{\ell_r}$ in the tableau 
          cannot exceed the number of entries $\ell-1$ in row $r$. 
        \end{enumerate}
      \item
        Given a Klein tableau $\Pi$, 
        we obtain the {\it underlying} LR-tableau $\Gamma$ by omitting the
        subscripts of those entries.  We say that $\Pi$ is a {\it refinement}
        of $\Gamma$. 
      \item 
        From an LR-tableau $\Gamma$ one can obtain a Klein tableau $\Pi$ by
        working through the entries in $\Gamma$ row by row 
        (starting at the top) and assigning
        to each symbol $\ell\geq2$ the largest available subscript 
        (due to the lattice permutation property, 
        there is always a subscript available).
        Then $\Pi$ is the {\it dominant} Klein tableau refining $\Gamma$. 
  \end{enumerate}
\end{defin}

We have seen in Section~\ref{section-from-ses} how a Klein tableau determines an
arc diagram.
The following result follows immediately.

\begin{lem}
  Suppose the Klein tableau $\Pi$ is a refinement of the LR-tableau $\Gamma$
  and has arc diagram $\Delta$.
  \begin{enumerate}
  \item At a vertex $m$ in $\Delta$, 
    the number of incoming arcs plus the number of poles can be read off
    as the number of entries $\singlebox1$ in row $m$ in $\Gamma$ or in $\Pi$.
  \item For each vertex $m$ in $\Delta$, the number of outgoing arcs
    equals the number of\/ $\singlebox 2$'s in row $m$ in $\Gamma$.
  \item Each box\/ $\singlebox{2_r}$ in row $m$ in $\Pi$
    corresponds to an arc from $m$ to $r$ in $\Delta$, and conversely.
  \item A Klein tableau of given type $(\alpha,\beta,\gamma)$ is determined uniquely
    by its arc diagram.
  \item The arc diagram associated with a dominant Klein tableau has no intersections. \qed
  \end{enumerate}
\end{lem}

Suppose that $\Pi$ and $\Pi'$ are two Klein tableaux
of partition type $(\alpha,\beta,\gamma)$ 
and that they are represented by arc diagrams $\Delta$ and $\Delta'$, respectively.
We define $\Pi \arcleq \Pi'$ if $\Delta \arcleq \Delta'$.
It follows from the lemma that the dominant Klein tableaux are minimal with respect
to this relation.

\subsection{A formula in Klein's paper}
%--------------------------------------

With a given LR-tableau $\Gamma$, we have associated the dominant Klein tableau $\Pi_0$;
it is special among all Klein tableaux refining
$\Gamma$ in the sense that the associated arc diagram has no intersections. 
The ``distance'' of an arbitrary Klein tableau $\Pi$ refining $\Gamma$ 
from the dominant one is introduced in \cite[Definition~1.4]{klein} as 
the {\it deviation $x(\Pi)$ from dominance} for which the following formula is given.

$$\begin{array}{rcl}x(\Pi) & = & 
  \displaystyle\sum_j \big(\bar\zeta_j - \bar\gamma_j + |\beta^{j-1}| - |\beta^j|\big)
               \sum_{k>j}\big(\bar\beta_k^{j-1} - \bar\zeta_k\big) \\[1ex]
        &  +& \displaystyle\sum_{j}\sum_{\ell>j}\big(\bar\beta^j_{\ell+1}-\bar\beta^{j-1}_{\ell+1}\big)
               \sum_{k=j+1}^\ell\big(\bar\beta_k^{j-1}-\bar\zeta_k\big).\end{array}$$

Here, a Klein tableau $\Pi$ of type $(\alpha,\beta,\gamma)$ 
with entries at most 2 is encoded by a sequence of partitions 
$$\Pi \;=\; [\gamma;\,\zeta=\beta^0;\,\beta^1,\ldots,\beta^s=\beta]$$
where $\gamma$ represents the region in the diagram $\Pi$ given by the empty boxes;
$\zeta=\beta^0$ is the partition for
the region marked off by the entries $\singlebox{}$ and $\singlebox1$;
and for each $j$, the partition $\beta^j$ represents region in $\Pi$
given by the empty boxes,
the boxes $\singlebox{1}$ and the boxes $\singlebox{2_i}$ where $i\leq j$.

\medskip
We can read off the deviation from dominance as the number of intersections
in the corresponding arc diagram:

\begin{lem}\label{lemma-crossings}
Let $\Pi$ be a Klein tableau with entries at most $2$.
The deviation $x(\Pi)$ from dominance equals the number $x(\Delta)$
of crossings in the arc diagram $\Delta$ which corresponds to $\Pi$.
\end{lem}

\begin{proof}
For each of the factors in the formula for $x(\Pi)$ we give an interpretation
in terms of data associated with the arc diagram $\Delta$:

\smallskip
In $\Pi$, the boxes $\singlebox1$ are located in the skew diagram $\zeta\backslash\gamma$,
so the number of boxes $\singlebox1$ in row $j$ is given by $\bar\zeta_j - \bar\gamma_j$.
Similarly, the boxes $\singlebox{2_j}$ are located  in the skew diagram 
$\beta^j\backslash\beta^{j-1}$, their number is $|\beta^{j}| - |\beta^{j-1}|$.
In $\Delta$, each box $\singlebox{2_j}$ corresponds to an arc ending at $j$,
so the difference $\bar\zeta_j - \bar\gamma_j + |\beta^{j-1}| - |\beta^j|$
counts the number of poles at position~$j$. 

\smallskip
As for the next factor, boxes of type $\singlebox{2_i}$ where $i<j$ are located 
in the skew diagram $\beta^{j-1}\backslash\zeta$; the number of such boxes in row $k$
is $\bar\beta_k^{j-1} - \bar\zeta_k$. The sum $\sum_{k>j}(\bar\beta_k^{j-1} - \bar\zeta_k)$ 
taken over all $k>j$ 
counts the number of arcs in $\Delta$ starting on the left of $j$ and ending on the right.
Thus the $j$-th term in the first sum counts the number of crossings with poles at 
position~$j$. 

\smallskip
Similarly, the skew diagram $\beta^j\backslash\beta^{j-1}$ consists exactly of the boxes
$\singlebox{2_j}$, so the factor $\bar\beta^j_{\ell+1}-\bar\beta^{j-1}_{\ell+1}$, which is the number
of such boxes in the $(\ell+1)$-st row, counts 
the arcs from $\ell+1$ to $j$.  

\smallskip We have already seen that $\beta^{j-1}\backslash\zeta$ is the skew diagram which
consists of the boxes $\singlebox{2_i}$ where $i<j$.  
Hence the sum $\sum_{k=j+1}^\ell(\bar\beta_k^{j-1}-\bar\zeta_k)$
counts the arcs which start in the interval $[j+1,\ell]$ and end
in $[1,j-1]$.  Those are the arcs which intersect with the arc from $\ell+1$ to $j$,
and are on the right hand side of it.
In conclusion, the term corresponding to $j$ and $\ell$ in the second sum
counts the number of intersections of arcs from $\ell+1$ to $j$ with arcs which are on
the right hand side of it.

\smallskip
We have seen that the first sum counts the intersections between arcs and poles in $\Delta$,
while the second sum counts the intersections of arcs with arcs.
\end{proof}

\subsection{A partial ordering on LR-tableaux}
%---------------------------------------------

\smallskip
We recall that there is a {\it partial ordering on partitions} given by
$\zeta\partleq\zeta'$ if for each natural number $\ell$, 
the inequality $\sum_{i=1}^\ell \zeta_i\leq \sum_{i=1}^\ell\zeta'_i$ holds.

\begin{defin}
Suppose $\Gamma$, $\Gamma'$ are LR-tableaux of partition type $(\alpha,\beta,\gamma)$,
both with entries at most 2. Representing LR-tableaux by increasing sequences of
partitions as in \cite{macd} or in the proof
of Lemma~\ref{lemma-crossings}, say, $\Gamma=[\gamma;\zeta;\beta]$ and 
$\Gamma'=[\gamma;\zeta';\beta]$, we define $\Gamma\partleq\Gamma'$
if $\zeta\partleq \zeta'$ holds.  
\end{defin}

This defines a partial ordering $\partleq$
on the set of LR-tableaux of a given partition type 
$(\alpha,\beta,\gamma)$ with $\alpha_1\leq 2$.
It is easy to see:

\begin{lem}\label{lemma-part-ordering}
Suppose $\Gamma$, $\Gamma'$ are LR-tableaux of the same partition type
and $\Delta$, $\Delta'$ are arc diagrams of LR-type $\Gamma$, $\Gamma'$, respectively.
\begin{enumerate}
\item If $\Delta<_{\rm arc}\Delta'$ via an arc operation of type {\bf (A)} or {\bf (B)}
  then $\Gamma=\Gamma'$.
\item If $\Delta<_{\rm arc}\Delta'$ via an arc operation of type {\bf(C)} or {\bf (D)}
  then $\Gamma>_{\rm part}\Gamma'$.  \qed
\end{enumerate}
\end{lem}

\subsection{Invariant subspaces}
%-------------------------------

Let $k$ be an~arbitrary field. For a partition $\alpha=(\alpha_1\geq\ldots\geq\alpha_n)$
we denote by $N_\alpha$ or by $N_\alpha(k)$ the finite dimensional $k[T]$-module
$N_\alpha(k)=k[T]/(T^{\alpha_1})\oplus\ldots\oplus k[T]/(T^{\alpha_n})$.
Note that $N_\alpha(k)$ can be considered as a~$k[T]/(T^{\alpha_1})$-module.
By $\mathcal N$ or $\mathcal{N}(k)$ 
we  denote this category of all 
nilpotent linear operators.
We write the objects of $\mathcal N$ as pairs $(V,\varphi)$
where $V$ is the underlying $k$-vector space and $\varphi:V\to V$
the nilpotent $k$-linear endomorphism given by multiplication by $T$.
If $(V,\varphi)$, $(V',\varphi')$ are objects in $\mathcal{N}$, then a~morphism 
$f:(V,\varphi)\to (V',\varphi')$ 
in $\mathcal N$
is a~linear map $f:V\to V'$ such that $\varphi'f=f\varphi$.

\smallskip
By $\Lambda$ we denote the $k$-algebra
$$ \Lambda=\left( \begin{array}{cc} k[T] &k[T] \\ 0&k[T] \end{array}\right).$$
Let $\mod(\Lambda)$ be the category of all finite dimensional right $\Lambda$-modules,
and $\mod_0(\Lambda)$ the full subcategory of $\mod(\Lambda)$ of all modules for which the
element $T={T\;0\choose 0\;T}$ acts nilpotently.
It is well known that objects in $\mod_0(\Lambda)$ 
may be identified with systems $(N_\alpha,N_\beta,f),$
where $\alpha$, $\beta$ are partitions and $f: N_\alpha\to N_\beta$ is 
a~$k[T]$-homomorphism.
Let  $A=(N_\alpha,N_\beta,f)$, $A'=(N_{\alpha'},N_{\beta'},f')$ be objects in
$\mod_0(\Lambda)$. 
A~morphism $\Psi:A\to A'$ is a~pair $(\psi_1,\psi_2)$, where
$\psi_1:N_\alpha\to N_{\alpha'}$, $\psi_2:N_\beta\to N_{\beta'}$ are homomorphisms
of $k[T]$-modules such that $f'\psi_1=\psi_2f$.  

\smallskip
Denote by $\mathcal S$ or $\mathcal{S}(k)$ the full subcategory of 
$\mod_0(\Lambda)$ consisting
of all systems $f=(N_\alpha,N_\beta,f)$,
where $f$ is a~monomorphism. 

\smallskip
For a natural number $n$,
we write $\mathcal S_n$ or $\mathcal S_n(k)$ for the full subcategory of $\mathcal S$
of all systems where the operator acts on the subspace with nilpotency index at most $n$.
Thus, the objects in $\mathcal S_2$ are the systems $(N_\alpha,N_\beta,f)$ where $\alpha_1\leq 2$.

\subsection{Pickets and bipickets}
%---------------------------------
\label{section-pickets}

The  category $\mathcal{S}_2(k)$ is of particular interest for us 
in this paper; it has discrete representation type:  Each indecomposable object is 
either isomorphic to a {\it picket} that is, it has the form 
$$P_\ell^m=(N_{(\ell)},N_{(m)},\iota)$$ 
where $0\leq\ell\leq\min\{2,m\}$ (so the ambient space $N_{(m)}$ has only one Jordan block,
and $N_{(\ell)}$ is the unique $T$-invariant subspace of dimension $\ell$),
or to a  {\it bipicket} 
$$B_2^{m,r}=(N_{(2)},N_{(m,r)},\delta)$$ 
where $1\leq r\leq m-2$
and $\delta:k[T]/(T^2)\to k[T]/(T^m)\oplus k[T]/(T^r)$ is given by
$\delta(1)=(T^{m-2},T^{r-1})$. 
Whenever we want to emphasize the dependence on the field $k$, we will write $P_\ell^m=P_\ell^m(k)$ and 
$B_2^{m,r}=B_2^{m,r}(k)$.

\medskip
For the definition of a Klein tableau for an object in $\mathcal S_2$
we refer to \cite{klein} or \cite{sch}, here we summarize by giving a 
brief description.
The tableau of a picket $P_\ell^m$ consists of a single column
of $m$ boxes; the $\ell$ boxes at the bottom carry the entries $1,\ldots,\ell$.
In case $\ell=2$, the entry $2$ has a subscript $2_r$ where $r=m-1$. 
The tableau for a bipicket $B_2^{m,r}$ has two columns of $m$ and $r$ boxes, 
which are aligned at the top.  The two boxes at the bottom are 
$\singlebox{2_r}$ and $\singlebox 1$, respectively.

\medskip
\begin{center}
\begin{tabular}{|c|c|c|c|c|}\hline
\multicolumn5{|c|}
             {\raisebox{-1ex}[0mm]{\bf Tableaux and diagrams 
                  for the objects in $\ind\mathcal S_2$}}
             \\[2ex] \hline
$X:$\raisebox{1ex}{\phantom!}\raisebox{-1.5ex}{\phantom!} & $P_0^m$ & $P_1^m$ & $P_2^m$ & $B_2^{m,r}$ \\ \hline
    \raisebox{12ex}[0mm]{$\Gamma(X):$} &
        \begin{picture}(8,12)(0,-15)
          \put(0,0)\smbox \put(0,6)\vdotbox
          \put(3,3){$\left.\makebox(0,5){}\right\}{\!}_m$}
        \end{picture}
    &   \begin{picture}(8,15)(0,-12)
          \put(0,1){\singlebox{1}}\put(0,4){\singlebox{}}\put(0,9)\vdotbox
          \put(3,5){$\left.\makebox(0,6){}\right\}{\!}_m$}
        \end{picture}
    &   \begin{picture}(8,18)(0,-9)
          \put(0,1){\singlebox{2}}\put(0,4){\singlebox{1}}
          \put(0,7){\singlebox{}}\put(0,12)\vdotbox
          \put(3,6){$\left.\makebox(0,8){}\right\}{\!}_m$}
        \end{picture}
    &   \begin{picture}(6,27)
        \put(-6,11){${}_{m\!\!}\left\{\makebox(0,13){}\right.$}
        \put(0,1){\singlebox{2}}\put(0,4){\singlebox{}}\put(0,9.3)\vdotbox
        \put(0,13){\singlebox{}}\put(0,16){\singlebox{}}
        \put(3,13){\singlebox{1}}\put(3,16){\singlebox{}}\put(0,21)\vdotbox
        \put(3,21)\vdotbox\put(6,17){$\left.\makebox(0,6){}\right\}{\!}_r$}
        \end{picture}
     \\ \hline
    \raisebox{12ex}[0mm]{$\Pi(X):$} &
        \begin{picture}(8,12)(0,-15)
          \put(0,0)\smbox \put(0,6)\vdotbox
        \put(3,3){$\left.\makebox(0,5){}\right\}{\!}_m$}
        \end{picture}
    &   \begin{picture}(8,15)(0,-12)
          \put(0,1){\singlebox{1}}\put(0,4){\singlebox{}}\put(0,9)\vdotbox
        \put(3,5){$\left.\makebox(0,6){}\right\}{\!}_m$}
        \end{picture}
    &   \begin{picture}(16,18)(0,-9)
          \put(0,4){\singlebox{1}}\put(0,7){\singlebox{}}\put(0,12)\vdotbox
        \put(0,1){\singlebox{2_r}}
        \put(3,6){$\left.\makebox(0,8){}\right\}{\!}_m$}
        \put(5,-3){$\scriptstyle r=m-1$}
        \end{picture}
    &   \begin{picture}(16,27)(-3,0)
        \put(-6,11){${}_{m\!\!}\left\{\makebox(0,13){}\right.$}
        \put(0,4){\singlebox{}}\put(0,1){\singlebox{2_r}}
        \put(0,9.3)\vdotbox
        \put(0,13){\singlebox{}}\put(0,16){\singlebox{}}\put(3,13){\singlebox{1}}\put(3,16){\singlebox{}}\put(0,21)\vdotbox
%       \put(0,15){\boxes20}\put(3,15){\boxes21}\put(0,21)\vdotbox
        \put(3,21)\vdotbox\put(6,17){$\left.\makebox(0,6){}\right\}{\!}_r$}
        \put(8,.6){$\scriptstyle r<m-1$}
        \end{picture}
     \\ \hline
$\Delta(X):$ & $\emptyset$ &
             $\begin{picture}(8,9)(0,3)
             \put(0,4){\line(1,0){8}}
             \put(3,3){$\bullet$}
             \put(4,4){\line(0,1){5}}
             \put(4,1){\makebox[0mm]{$\scriptstyle m$}}
             \end{picture}$
          &  
             $\begin{picture}(16,9)(0,3)
             \put(0,4){\line(1,0){16}}
             \put(3,3){$\bullet$}
             \put(11,3){$\bullet$}
             \put(8,4){\oval(8,8)[t]}
             \put(4,1){\makebox[0mm]{$\scriptstyle m$}}
             \put(12,1){\makebox[0mm]{$\scriptstyle m-1$}}
             \end{picture}$
          &  
             $\begin{picture}(16,9)(0,3)
             \put(0,4){\line(1,0){16}}
             \put(3,3){$\bullet$}
             \put(8,4){\makebox[0mm]{$\cdots$}}
             \put(11,3){$\bullet$}
             \put(8,4){\oval(8,8)[t]}
             \put(4,1){\makebox[0mm]{$\scriptstyle m$}}
             \put(12,1){\makebox[0mm]{$\scriptstyle r$}}
             \end{picture}$ \\[2ex] \hline
\end{tabular}
\end{center}

  The Klein tableau for a direct sum 
$M\oplus M'$ has a diagram given by the union $\beta\cup\beta'$ of the 
partitions representing the ambient spaces, and in each row the entries are
obtained by lexicographically ordering the entries in the corresponding rows
in the tableaux for $M$ and $M'$, with empty boxes coming first. 

\smallskip\noindent{\it Example:} The Klein tableaux $\Pi$ in the example
in Section~\ref{section-from-ses} is given by
$$X\;=\;B_2^{5,3}\oplus B_2^{4,2}\oplus P_1^3\oplus P_1^1.$$

We recall from \cite[Proposition 2]{sch}:

\begin{thm}For any field $k$, 
there is a one-to-one correspondence between the isomorphism classes 
of objects in $\mathcal S_2(k)$ and the Klein tableaux with entries at most $2$.
\qed
\label{thm-pic-bipic}\end{thm}

The object in $\mathcal S_2(k)$ which corresponds to the Klein tableau $\Pi$ will be 
denoted by  $M_\Pi$ or $M_\Pi(k)$.
It is straightforward to recover the indecomposable summands 
of $M_\Pi$ from $\Pi$. 
To start, note that each entry $\singlebox{2_r}$ in row $m$ gives
rise to a summand of type $B_2^{m,r}$ if $r\leq m-2$ or $P_2^m$ if $r=m-1$.

\section{Partial orders on Klein tableaux}\label{sec-partial-orders}
%======================================

For partitions $\alpha$, $\beta$, $\gamma$, denote by $\mathcal S_\alpha^\beta$ 
the full subcategory of $\mathcal S$ of all embeddings 
of type $f:N_\alpha\to N_\beta$,
and by $\mathcal S_{\alpha,\gamma}^\beta$ the full subcategory of $\mathcal S_\alpha^\beta$
of those embeddings which satisfy in addition that $\Coker f\cong N_\gamma$ holds.

\smallskip
We introduce partial orders $\extleq$, $\degleq$, $\homleq$ and $\arcleq$
for objects in $\mathcal S_{\alpha,\gamma}^\beta$ and show the implications 
$\extleq\;\Longrightarrow\;\degleq\;\Longrightarrow\;\homleq.$
Using two results from Section~\ref{chapter-hom-tab},
we can complete the proofs of Theorems~\ref{thm-first-main} and \ref{thm-third-main}.

\smallskip
Whenever we deal with the degeneration order $\degleq$,
we will assume that $k$ is an algebraically closed field.

\subsection{Four partial orders}\label{sec-four-partial}
%-------------------------------
We consider the affine variety $H_\alpha^\beta(k)=\Hom_k(N_\alpha(k),N_\beta(k))$
(consisting of all $|\beta|\times |\alpha|-$matrices with coefficients in $k$). 
On $H_\alpha^\beta(k)$ we consider the Zariski topology and on all subsets of
$H_\alpha^\beta(k)$ we work with the induced topology. Let $V_\alpha^\beta(k)$
be the subset of $H_\alpha^\beta(k)$ consisting of all matrices that define a~monomorphism
in the category $\mathcal{N}(k)$. Note that $V_\alpha^\beta(k)$
is a~locally closed subset of $H_\alpha^\beta(k)$. On $V_\alpha^\beta(k)$
acts the algebraic group $\Aut_{\mathcal{N}}(N_\alpha(k))\times \Aut_{\mathcal{N}}( N_\beta(k))$
via $(g,h)\cdot f=hfg^{-1}$. Orbits of this action correspond bijectively to isomorphism
classes of objects in $\mathcal S_\alpha^\beta$.
For a map $f:N_\alpha(k)\to N_\beta(k)$, denote by $G_f$ the orbit of $f$ in
$V_\alpha^\beta(k)$.
Let
$V_{\alpha,\gamma}^\beta(k)$ be the set of monomorphisms $f$ in
 $V_\alpha^\beta(k)$ such that 
$(N_\alpha(k),N_\beta(k),f)\in\mathcal S_{\alpha,\gamma}^\beta(k)$.

\smallskip
Let $Y=(N_\alpha,N_\beta,f)$ and $Z=(N_\alpha,N_\beta,g)$ be objects 
in $\mathcal S_{\alpha,\gamma}^\beta(k)$.

\begin{itemize}
 \item The relation $Y \leq_{\rm ext} Z$  holds if there exist
a natural number $s$,  
objects  $M_i$, $U_i$, $V_i$ in $\mathcal{S}(k)$ and short exact
sequences $0\to U_i\to M_i\to V_i\to 0$ in~$\mathcal{S}(k)$  
such that $Y\cong M_1$, $U_i\oplus V_i\cong M_{i+1}$ for $1\leq i\leq s$, 
and $Z\cong M_{s+1}$.

\item The relation $Y \leq_{\rm deg} Z$ 
holds if $G_g \subseteq \overline{G_f}$ in $V_{\alpha,\gamma}^\beta(k)$,
where $\overline{G_f}$ is the closure of $G_f$.

\item The relation $Y\leq_{\rm hom} Z$  holds
if $$[X,Y]\leq [X,Z]$$ for any object $X$ in $\mathcal{S}(k)$.
Here we write $[X,Y]=\dim_k\Hom_{\Lambda}(X,Y)$ for $\Lambda$-modules $X,Y$. 

\item If the objects $Y$, $Z$ are in $\mathcal S_2$, 
then the relation $Y\leq_{\rm arc} Z$ holds if $\Pi\leq_{\rm arc}\Pi'$, where $\Pi,\Pi'$
are the Klein tableaux associated with  $Y$ and $Z$, 
respectively (see Theorem~\ref{thm-pic-bipic}).
\end{itemize}

For $\Gamma$ an LR-tableau of type $(\alpha,\beta,\gamma)$, 
let $V_\Gamma(k)$ denote the subset of $V_{\alpha,\gamma}^\beta(k)$ consisting of all
$f$ with LR-tableau $\Gamma$ (see \cite[Chapter~2]{sch}), and define
$\mathcal S_\Gamma$ correspondingly.
The partial orders $\leq_{\rm ext}$, $\leq_{\rm deg}$, $\leq_{\rm hom}$, and if applicable $\arcleq$,
can be defined via restriction for the objects in $\mathcal S_\Gamma$.

\subsection{Modules versus embeddings}
%-------------------------------------

The three lemmata in this section show that we can work in the module category $\mod(\Lambda)$ 
in order to  decide the relations $\extleq$, $\degleq$ and $\homleq$.

\begin{lem}\label{lem-ext-order}
  The relation   $Y\extleq Z$  
  holds if and only if there exist a natural number $s$, 
  objects  $M_i$, $U_i$, $V_i$ in $\mod(\Lambda)$ and short exact
  sequences $0\to U_i\to M_i\to V_i\to 0$ in~$\mod(\Lambda)$  
  such that $Y\cong M_1$, 
    $U_i\oplus V_i\cong M_{i+1}$ for $1\leq i\leq s$,
    and $Z\cong M_{s+1}$ hold.
\end{lem}

\begin{proof} Assume that there exist $\Lambda$-modules $M_i$, $U_i$, $V_i$
satisfying the required conditions.  Since $M_{s+1}\cong
U_s\oplus V_s$, $M_{s+1}\in \mathcal{S}(k)$ and since the category $\mathcal{S}(k)$ is closed under
direct summands, the $\Lambda$-modules $U_s$, $V_s$ are in $\mathcal{S}(k)$.
Moreover there exists an exact sequence $$0\to U_s\to
U_{s-1}\oplus V_{s-1}\to V_s\to 0 $$ and therefore $U_{s-1}$,
$V_{s-1}$ are in $\mathcal{S}(k)$, because the category $\mathcal{S}(k)$  is
closed under extensions. Continuing this way we prove that all
$U_i$, $V_i$ are in $\mathcal{S}(k)$ and consequently 
$Y\extleq Z$. 
Since the converse implication is obvious, we are done.
\end{proof}

  \begin{lem}\label{lem-deg-order}
    \begin{enumerate}
    \item For points $Y=(N_\alpha,N_\beta,f)$, $Z=(N_\alpha,N_\beta,g)$ in 
      $V_{\alpha,\gamma}^\beta(k)$,
      the relation $Y\degleq Z$
      holds if and only if $G_g \subseteq \overline{G_f}^a$, 
      where $\overline{G_f}^a$ is the closure of $G_f$ 
      in the variety $M_\alpha^\beta(k)$ of $\Lambda$-modules.
    \item If\/ $\Gamma$ is an LR-tableau of type $(\alpha,\beta,\gamma)$
      and $Y=(N_\alpha,N_\beta,f)$, $Z=(N_\alpha,N_\beta,g)$ are points in $V_\Gamma(k)$,
      then the relation $Y\degleq Z$
      holds if and only if $G_g \subseteq \overline{G_f}^a$.
    \end{enumerate}
  \end{lem}

\begin{proof} 
  We only show the first statement, the proof of the second statement is similar.
  The variety $M_\alpha^\beta(k)$ of $\Lambda$-modules is defined as the 
  (closed) subset of
  $H_\alpha^\beta(k)$ consisting of those linear maps which define $k[T]$-homomorphisms.
  If $(N_\alpha,N_\beta,f) \leq_{\rm deg} (N_\alpha,N_\beta,g)$, 
  then of course $G_g \subseteq \overline{G_f}^a$ 
  (because $\overline{G_f}\subseteq \overline{G_f}^a$).

  Conversely, let $G_g \subseteq \overline{G_f}^a$. 
  Since $\overline{G_f}=\overline{G_f}^a\cap
  V_{\alpha,\gamma}^\beta(k)$ 
    and $G_g\subseteq V_{\alpha,\gamma}^\beta(k)$ we have
    $G_g\subseteq \overline{G_f}$. 
    Thus $(N_\alpha,N_\beta,f) \leq_{\rm deg} (N_\alpha,N_\beta,g)$.
\end{proof}

\begin{lem}\label{lem:hom-order} Suppose $Y,Z\in\mathcal S_2$
    have the same partition type.
\begin{enumerate}
\item   The relation $Y\homleq Z$ holds
if and only if  $[X,Y]\leq[X,Z]$ holds
for any object $X$ in $\mathcal{S}_2(k)$.
\item  The relation $Y\homleq Z$ holds
if and only if  $[X,Y]\leq[X,Z]$ holds 
for any $\Lambda$-module $X$.
\end{enumerate}
\end{lem}

\begin{proof}
We show the second statement first.

\smallskip
2. It is enough to prove that if $[X',Y]\leq [X',Z]$ 
for any object $X'$ in $\mathcal{S}$, then $[X,Y]\leq [X,Z]$ 
for any module $X$ in $\mod(\Lambda)$.
 
Let $X=(X_1,X_2,h)\in \mod(\Lambda)$ and write $Y=(Y_1,Y_2,f)$ where 
$f$ is a~monomorphism. Consider
the object $X'=(X_1',X_2,h')\in \mathcal{S}$,
where $X_1'=X_1/\Ker\, h$ and $h':X_1'\to X_2$ is induced by $h$. Let 
$a=(a_1,a_2):X\to Y$ be a~morphism. Note that if $x\in \Ker\, h$, then $x\in \Ker\, a_1$
because $f$ is a~monomorphism. Therefore there exists $a_1':X_1'\to Y_1$
such that $a_1=a_1'\circ {\rm can}_1$, where ${\rm can}_1: X_1\to X_1'$ is the canonical epimorphism.
 It is easy to see that the pair $a'=(a_1',a_2)$ defines a~morphism $X'\to Y$.
Writing ${\rm can} = ({\rm can}_1,1): X\to X'$, this morphism satisfies $a=a'\circ{\rm can}$.

Thus, ${\rm can}: X\to X'$ is a left approximation for $X$
in $\mathcal S$.  Since ${\rm can}$ is onto, it follows that $[X,Y]=[X',Y]$.

\smallskip
1. Next we show that if $[X'',Y]\leq [X'',Z]$ 
for any object $X''$ in $\mathcal{S}_2$, then $[X',Y]\leq 
[X',Z]$ for any object $X'$ in $\mathcal{S}$. 

Now let $X'=(X'_1,X'_2,h')\in \mathcal S$
  and $Y=(Y_1,Y_2,f)\in \mathcal{S}_2$, so $T^2Y_1=0$. 
Consider the object $X''=(X_1'',X_2'',h'')\in \mathcal{S}_2$,
where $X_1''=X'_1/T^2X'_1$, $X_2''=X_2'/h'(T^2X'_1)$  and $h'':X_1''\to X_2''$ is 
the map induced by $h'$.
Since $h'$ is a monomorphism, so is $h''$.  Let ${\rm can}=({\rm can}_1,{\rm can}_2):X'\to X''$
be the canonical map.  We show that a morphism 
$a=(a_1,a_2):X'\to Y$ factors over ${\rm can}$.  
Since $T^2Y_1=0$, we have $T^2X'_1\subseteq \Ker\, a_1$,
so $a_1$ factors over ${\rm can}_1$:
Write $a_1=a_1''\circ{\rm can}_1$. 
Since $h'(T^2X'_1)\subseteq \Ker\, a_2$, the map $a_2$ factors over ${\rm can}_2: X'_2\to X''_2$:
There is $a_2''$ with $a_2=a_2''\circ{\rm can}_2$. 
Since ${\rm can}_1$ is onto, the pair $a''=(a_1'',a_2'')$
is a morphism from $X''$ to $Y$ in $\mathcal S$ satisfying $a=a''\circ{\rm can}$.

We have seen that ${\rm can}: X'\to X''$ is a left approximation for $X'$ in $\mathcal S_2$;
since ${\rm can}$ is onto, it follows that $[X',Y]=[X'',Y]$.
We are done.
\end{proof}

\subsection{The partial orders are equivalent}
%---------------------------

We can now complete the proofs of Theorem~\ref{thm-first-main} 
and Theorem~\ref{thm-third-main}, up to two results about the arc order which are
shown in the next section.

\smallskip
We restate both theorems to include statements about arbitrary fields.

\begin{thm}\label{thm-main} Let $k$ be an~arbitrary field and assume that
$Y,Z\in\mathcal{S}_2(k)$ have the same partition type $(\alpha,\beta,\gamma)$.
The following conditions are equivalent
\begin{enumerate}
 \item $Y\leq_{\rm arc} Z$,
 \item $Y\leq_{\rm ext} Z$,
 \item $Y\leq_{\rm hom} Z$.
\end{enumerate}
If in addition the field $k$ is algebraically closed, 
then the conditions stated above are equivalent with
\begin{enumerate}
 \item[4.] $Y\leq_{\rm deg} Z$.
\end{enumerate}
\end{thm}

\begin{proof}
Applying the functor
$\Hom_k(X,-)$ to the short exact sequences given in the definition of $\leq_{\rm ext}$,
it is easy to see that $Y\leq_{\rm ext} Z$ implies $Y\leq_{\rm hom} Z.$

If $k$ is an~algebraically closed field, then by \cite{bongartz,riedtmann}, 
Lemmata \ref{lem-deg-order}, \ref{lem-ext-order} and \ref{lem:hom-order}
we have 
$$Y\leq_{\rm ext} Z \;\Longrightarrow\; Y\leq_{\rm deg} Z\;\Longrightarrow\; Y\leq_{\rm hom} Z.$$ 
The implications
$$Y\leq_{\rm hom} Z \;\Longrightarrow\; Y\leq_{\rm arc} Z\;\Longrightarrow\; Y\leq_{\rm ext} Z$$ 
independently on $k$ follow
from Theorem \ref{theorem-hom-arc} and Lemma \ref{lem-tab-ext}, respectively.
\end{proof}

\begin{cor}\label{corollary-hom-arc}
  Let $k$ be an arbitrary field.
  Suppose $Y$ and $Z$ are objects in $\mathcal S_2(k)$ corresponding to
  the same LR-tableau.  Then
  $$Y \leq_{\rm hom} Z \;\Longleftrightarrow\; Y\leq_{\rm arc} Z 
      \;\Longleftrightarrow\; Y\extleq Z.$$
  In this case it is possible to convert the arc diagram for $Z$ into the 
  arc diagram for $Y$ using only operations of type {\bf (A)} and  {\bf (B)}.

  \smallskip
  If $k$ is an algebraically closed field, then the above conditions are equivalent
  to $Y\degleq Z$.
\end{cor}

\begin{proof}
Suppose $Y,Z$ are objects in $\mathcal S_2$ corresponding to the same LR-tableau $\Gamma$.
Since $Y,Z$ have the same partition type, we obtain from Theorem \ref{thm-lattice} that
the partial orders coincide.
We have seen in Lemma~\ref{lemma-part-ordering} 
that arc operations of type {\bf (C)} or {\bf (D)}
change the LR-type of the module in one direction.  
As a consequence, if $Y\homleq Z$ holds then
the arc diagram for $Y$ can be obtained from the arc diagram for $Z$ 
by using only arc operations of type {\bf (A)} or {\bf (B)}.
In case the field $k$ is algebraically closed, the above proof where
we use the second part in Lemma~\ref{lem-deg-order} shows that 
$$Y\leq_{\rm ext} Z \;\Longrightarrow\; Y\leq_{\rm deg} Z\;\Longrightarrow\; Y\leq_{\rm hom} Z.$$ 
\end{proof}

\section{When are two arc diagrams in $\leq_{\rm arc}$-relation?}
%=============================================================
\label{chapter-hom-tab}

Our aim is to show 
that $Y \leq_{\rm hom} Z$ implies $Y\leq_{\rm arc} Z$.

\begin{thm}\label{theorem-hom-arc}
  Suppose the objects $Y$ and $Z$ in $\mathcal S_2(k)$ have 
  the same partition type.
    If $Y \leq_{\rm hom} Z$ holds, then so does $Y\leq_{\rm arc} Z$.
\end{thm}

Consider the following example.

$$
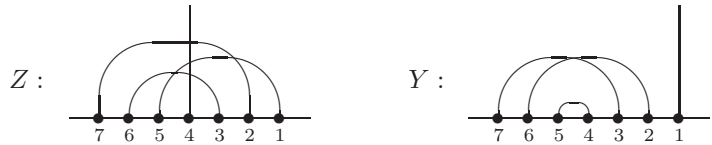
\begin{picture}(32,15)(0,3)
\put(-8,8){$Z:$}
\put(0,4){\line(1,0){32}}
\multiput(3,3)(4,0)7{$\bullet$}
\put(14,4){\oval(20,20)[t]}
\put(20,4){\oval(16,16)[t]}
\put(14,4){\oval(12,12)[t]}
\put(16,4){\line(0,1){15}}
\put(4,1){\makebox[0mm]{$\scriptstyle 7$}}
\put(8,1){\makebox[0mm]{$\scriptstyle 6$}}
\put(12,1){\makebox[0mm]{$\scriptstyle 5$}}
\put(16,1){\makebox[0mm]{$\scriptstyle 4$}}
\put(20,1){\makebox[0mm]{$\scriptstyle 3$}}
\put(24,1){\makebox[0mm]{$\scriptstyle 2$}}
\put(28,1){\makebox[0mm]{$\scriptstyle 1$}}
\end{picture}
\qquad\qquad\qquad
\begin{picture}(32,15)(0,3)
\put(-8,8){$Y:$}
\put(0,4){\line(1,0){32}}
\multiput(3,3)(4,0)7{$\bullet$}
\put(12,4){\oval(16,16)[t]}
\put(16,4){\oval(16,16)[t]}
\put(14,4){\oval(4,4)[t]}
\put(28,4){\line(0,1){15}}
\put(4,1){\makebox[0mm]{$\scriptstyle 7$}}
\put(8,1){\makebox[0mm]{$\scriptstyle 6$}}
\put(12,1){\makebox[0mm]{$\scriptstyle 5$}}
\put(16,1){\makebox[0mm]{$\scriptstyle 4$}}
\put(20,1){\makebox[0mm]{$\scriptstyle 3$}}
\put(24,1){\makebox[0mm]{$\scriptstyle 2$}}
\put(28,1){\makebox[0mm]{$\scriptstyle 1$}}
\end{picture}
$$

It turns out that $Y\leq_{\rm hom}Z$. We want to show that 
$Y\leq_{\rm arc}Z$.
In order to satisfy the condition in the definition of the $\leq_{\rm arc}$-order
we need to exhibit a sequence of steps of type
  {\bf (A)}, {\bf (B)}, {\bf (C)} and {\bf (D)},
which transform the arc diagram for $Y$ into the arc diagram for $Z$.
This does  not seem to be entirely trivial, since even inviting
moves can destroy the $\leq_{\rm hom}$-relation.  Consider for example
the operation of the type {\bf (B)}
which replaces the arc from 5 to 1 and the pole at 4 in $Z$ by
the arc from 5 to 4 and a pole at 1.  This yields $\tilde Z$:

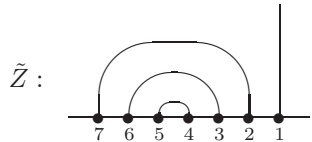
$$
\begin{picture}(32,15)(0,3)
\put(-8,8){$\tilde Z:$}
\put(0,4){\line(1,0){32}}
\multiput(3,3)(4,0)7{$\bullet$}
\put(14,4){\oval(20,20)[t]}
\put(14,4){\oval(12,12)[t]}
\put(14,4){\oval(4,4)[t]}
\put(28,4){\line(0,1){15}}
\put(4,1){\makebox[0mm]{$\scriptstyle 7$}}
\put(8,1){\makebox[0mm]{$\scriptstyle 6$}}
\put(12,1){\makebox[0mm]{$\scriptstyle 5$}}
\put(16,1){\makebox[0mm]{$\scriptstyle 4$}}
\put(20,1){\makebox[0mm]{$\scriptstyle 3$}}
\put(24,1){\makebox[0mm]{$\scriptstyle 2$}}
\put(28,1){\makebox[0mm]{$\scriptstyle 1$}}
\end{picture}
$$

Now however, we have gone too far as $\tilde Z\leq_{\rm arc}Y$
by a move of type {\bf (A)}, and also $\tilde Z\leq_{\rm hom}Y$
(consider the functor $\Hom(B_2^{7,3},-)$).  We will come back to this 
example in Section~\ref{section-example}.

\medskip
The strategy for the proof of Theorem~\ref{theorem-hom-arc}
is to gain thorough understanding on homomorphisms between 
objects in $\mathcal S_2(k)$.  Recall that $Y\leq_{\rm hom}Z$
if and only if for each indecomposable $X$, the integer
$$\delta H(Y,Z)_X\;=\;[X,Z]-[X,Y]$$
is nonnegative, where $[A,B]=[A,B]_{\mathcal S}=
\dim_k\Hom_{\mathcal S}(A,B)$. 

\smallskip
In the first section we will give explicit formulae for the
dimensions of the homomorphism spaces between objects in $\mathcal S_2$.

\smallskip
Then in Section~\ref{section-operations} 
we show how those dimensions change when $Y$ 
is replaced by $Z$, the module obtained from $Y$ 
by performing an operation of type 
{\bf (A)}, {\bf (B)}, {\bf(C)} or {\bf (D)}
on the arc diagram. 

\smallskip
We analyze the Hom-matrix $\delta H(Y,Z)$
which forms the basis of our construction of moves in the 
$\leq_{\rm arc}$-direction. In particular, we can read off from this matrix
the multiplicities of indecomposable objects as direct summands
in $Y$ and in $Z$, respectively.

\smallskip
The example in Section~\ref{section-example} illustrates how the 
decomposition of the Hom-matrix as a sum of characteristic functions
of suitable parallelograms translates into the desired sequence of
$\leq_{\rm arc}$-moves. 

\smallskip
In the last section we conclude the proof of Theorem~\ref{theorem-hom-arc}.
For this we show in Proposition~\ref{proposition-key-lemma}
that whenever the Hom-matrix $\delta H(Y,Z)$ is nonnegative
and nonzero, then there exists an $\arcleq$-move which leads to a new matrix
with smaller nonnegative entries.

\subsection{The category $\mathcal S_2(k)$}\label{section-dim-hom}
%------------------------------------------

Denote by $\mathcal S_2^n(k)$ the full subcategory of $\mathcal S_2(k)$
of all objects where the operator acts with nilpotency index at most $n$.
We have seen in \cite[Section~3.2]{s-b} that $\mathcal S_2^n(k)$ is an
exact Krull-Remak-Schmidt category with Auslander-Reiten sequences.  
Interestingly, the Auslander-Reiten sequences starting at a given object
do not depend on the choice of $n$, provided this number is large enough,
and hence are Auslander-Reiten sequences in the category $\mathcal S_2(k)$.
Regarding Auslander-Reiten sequences ending at a given object, the
dual result holds with the exception of 
the objects of type $P_1^r$ which occur as end terms
of the following Auslander-Reiten sequences in $\mathcal S_2^n(k)$
$$0 \longrightarrow B_2^{n,r-1}\longrightarrow B_2^{n,r}\oplus P_1^{r-1}
    \longrightarrow P_1^r\longrightarrow 0$$
($2\leq r\leq n-2$).  

\smallskip
The Auslander-Reiten quiver for each of the categories $\mathcal S_2^n(k)$
is obtained by identifying the objects of type $P_1^r$ on the left with their
counterparts on the
right in the following picture, thus yielding a Moebius band.

$$
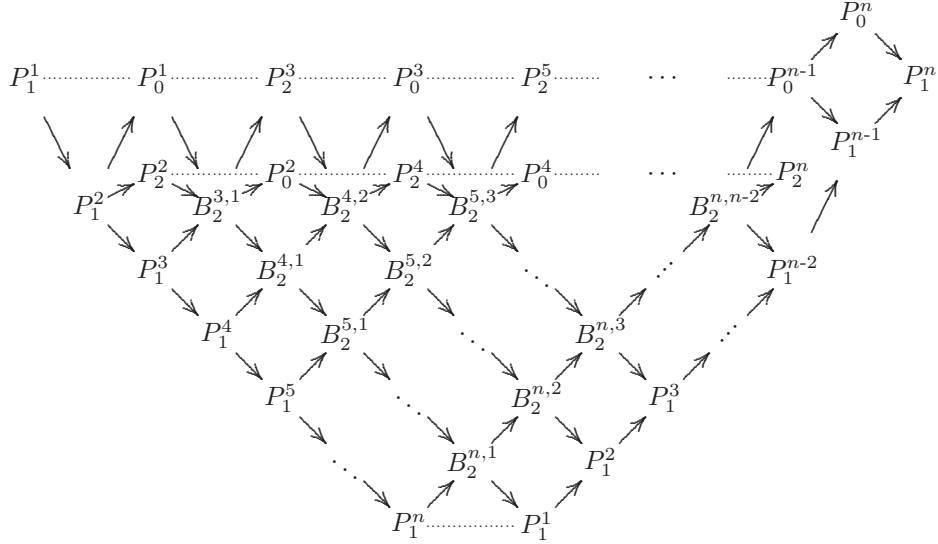
\beginpicture\setcoordinatesystem units <.85cm,.85cm>
\setcounter{boxsize}{2}
\put {} at 0 6.5
\put {} at 0 -.8
%\put{$\Gamma_{\mathcal S_2}:$} at 0.5 2
\setlength\unitlength{1mm}
\put{$P_1^1$} at 0 6
%\put{$\times$} at -.5 6
\put{$P_1^2$} at 1 4
%\put{$\times$} at .5 4
\put{$P_0^1$} at 2 6
\put{$P_2^2$} at 2 4.5
\put{$P_1^3$} at 2 3
%\put{$\times$} at 1.5 3
\put{$B_2^{3,1}$} at 3 4
\put{$P_1^4$} at 3 2
%\put{$\times$} at 2.5 2
\put{$P_2^3$} at 4 6
\put{$P_0^2$} at 4 4.5
%\put{{\bf (*)}} at 4 4
\put{$B_2^{4,1}$} at 4 3
\put{$P_1^5$} at 4 1
%\put{$\times$} at 3.5 1
\put{$B_2^{4,2}$} at 5 4
\put{$B_2^{5,1}$} at 5 2
%\put{$P_1^6$} at 5 0
%\put{$\times$} at 4.5 0
\put{$P_0^3$} at 6 6
\put{$P_2^4$} at 6 4.5
\put{$B_2^{5,2}$} at 6 3
%\put{$B_2^{6,1}$} at 6 1
\put{$P_1^n$} at 6 -1
\put{$B_2^{5,3}$} at 7 4
%\put{$B_2^{6,2}$} at 7 2
\put{$B_2^{n,1}$} at 7 0
\put{$P_2^5$} at 8 6
\put{$P_0^4$} at 8 4.5
%\put{$B_2^{6,3}$} at 8 3
\put{$B_2^{n,2}$} at 8 1
\put{$P_1^1$} at 8 -1
%\put{$B_2^{6,4}$} at 9 4
\put{$B_2^{n,3}$} at 9 2
\put{$P_1^2$} at 9 0
%\put{$P_2^6$} at 10 4.5
%\put{$P_0^5$} at 10 6
\put{$P_1^3$} at 10 1
%\put{$\cdots$} at 11 6
%\put{$P_1^4$} at 11 2
\put{$B_2^{n,n\text{-}2}\;$} at 11 4
\put{$P_1^{n\text-2}$} at 12 3
\put{$P_2^n$} at 12 4.5
\put{$P_0^{n\text-1}$} at 12 6
\put{$P_0^n$} at 13 7
\put{$P_1^{n\text{-}1}$} at 13 5
\put{$P_1^n$} at 14 6
%\multiput{$\cdot$} at 12.9 3.9  13 4  13.1 4.1 /
\arr{.3 5.4}{.7 4.6}
\arr{1.3 4.6}{1.7 5.4}
\arr{1.3 4.15}{1.7 4.35}
\arr{1.3 3.7}{1.7 3.3}
\arr{2.3 5.4}{2.7 4.6}
\arr{2.3 4.35}{2.7 4.15}
\arr{2.3 3.3}{2.7 3.7}
\arr{2.3 2.7}{2.7 2.3}
\arr{3.3 4.6}{3.7 5.4}
\arr{3.4 4.2}{3.7 4.35}
\arr{3.3 3.7}{3.7 3.3}
\arr{3.3 2.3}{3.7 2.7}
\arr{3.3 1.7}{3.7 1.3}
\arr{4.3 5.4}{4.7 4.6}
\arr{4.3 4.35}{4.7 4.15}
\arr{4.4 3.4}{4.7 3.7}
\arr{4.3 2.7}{4.7 2.3}
\arr{4.3 1.3}{4.7 1.7}
\arr{4.3 .7}{4.7 .3}
\arr{5.3 4.6}{5.7 5.4}
\arr{5.4 4.2}{5.7 4.35}
\arr{5.3 3.7}{5.7 3.3}
\arr{5.3 2.3}{5.7 2.7}
\arr{5.3 1.7}{5.7 1.3}
%\arr{5.3 .3}{5.7 .7}
\arr{5.3 -.3}{5.7 -.7}
\arr{6.3 5.4}{6.7 4.6}
\arr{6.3 4.35}{6.7 4.15}
\arr{6.4 3.4}{6.7 3.7}
\arr{6.3 2.7}{6.7 2.3}
%\arr{6.3 1.3}{6.7 1.7}
\arr{6.3 -.7}{6.7 -.3}
\arr{6.3 .7}{6.7 .3}
\arr{7.3 4.6}{7.7 5.4}
\arr{7.4 4.2}{7.7 4.35}
\arr{7.3 3.7}{7.7 3.3}
%\arr{7.4 2.4}{7.7 2.7}
\arr{7.3 1.7}{7.7 1.3}
\arr{7.3 .3}{7.7 .7}
\arr{7.3 -.3}{7.7 -.7}
%\arr{8.3 5.4}{8.7 4.6}
%\arr{8.3 4.35}{8.7 4.15}
%\arr{8.4 3.4}{8.7 3.7}
\arr{8.3 2.7}{8.7 2.3}
\arr{8.3 1.3}{8.7 1.7}
\arr{8.3 .7}{8.7 .3}
\arr{8.3 -.7}{8.7 -.3}
%\arr{9.4 4.8}{9.7 5.4}
%\arr{9.4 4.2}{9.7 4.35}
%\arr{9.3 3.7}{9.7 3.3}
\arr{9.3 2.3}{9.7 2.7}
\arr{9.3 1.7}{9.7 1.3}
\arr{9.3 .3}{9.7 .7}
%\arr{10.3 5.5}{10.7 4.7}
\arr{10.3 3.3}{10.7 3.7}
\arr{10.3 1.3}{10.7 1.7}
\arr{11.5 4.25}{11.7 4.35}
\arr{11.3 4.6}{11.7 5.4}
\arr{11.3 3.7}{11.7 3.3}
\arr{11.3 2.3}{11.7 2.7}
\arr{12.3 3.6}{12.7 4.4}
\arr{12.3 5.7}{12.7 5.3}
\arr{12.3 6.3}{12.7 6.7}
\arr{13.3 6.7}{13.7 6.3}
\arr{13.3 5.3}{13.7 5.7}
\setdots<2pt>
\plot 0.3 6  1.7 6 /
\plot 2.3 6  3.7 6 /
\plot 4.3 6  5.7 6 /
\plot 6.3 6  7.7 6 /
\plot 8.3 6  9 6 /
\plot 11 6  11.7 6 /
\plot 2.3 4.5  3.7 4.5 /
\plot 4.3 4.5  5.7 4.5 /
\plot 6.3 4.5  7.7 4.5 /
\plot 8.3 4.5  9 4.5 /
\plot 11 4.5  11.7 4.5 /
\plot 6.3 -1  7.7 -1 /
\multiput{$\ddots$} at 8 3  7 2  6 1  5 0 /
\multiput{$\cdots$} at 10 6  10 4.5 /
\multiput{$\cdot$} at 9.9 2.9  10 3  10.1 3.1     10.9 1.9  11 2  11.1 2.1 /
\endpicture
$$

For each pair $(X,Y)$ of indecomposable objects in $\mathcal S_2(k)$ we determine
in the table below
the dimension of the $k$-space $\Hom_{\mathcal S}(X,Y)$. Most of the numbers are taken 
from \cite[Lemma 4]{sch}.

\begin{figure}[tbh]
\begin{center}
\begin{tabular}{|r|@{}c@{}|@{}c@{}|@{}c@{}|@{}c@{}|}\hline
\multicolumn{5}{|c|}{\bf Dimensions of Spaces $\Hom(X,Y)$, $X,Y\in \ind\mathcal S_2(k)$}\\
\hline
  $X$ &  $Y=P_0^m$  &  $P_2^m$  &  $B_2^{m,r}$  &  $P_1^m$  \\ 
\hline \hline
$P_0^\ell$ & $\min\{\ell,m\}$ & $\min\{\ell,m\}$ 
            & $\begin{array}{l} \min\{\ell,m\} \\ + \min\{\ell,r\} \end{array}$ 
            & $\min\{\ell,m\}$ \\
\hline
$P_2^\ell$      & $\min\{\ell-2,m\}$
               & $\min\{\ell,m\}$
               & $\begin{array}{l} \min\{\ell-1,m\} \\ +\min\{\ell-1,r\} \end{array}$
               & $\min\{\ell-1,m\}$ \\
\hline
$B_2^{\ell,t}$ & $\begin{array}{l}  \min\{\ell-1,m\} \\ + \min\{t-1,m\} \end{array}$
             & $\begin{array}{l}  \min\{\ell,m\} \\ + \min\{t,m\} \end{array}$
             & $\begin{array}{l}  \min\{\ell-1,m\} \\ + \min\{t,m\} \\ 
                        + \min\{\ell-1,r\} \\ + \min\{t,r\} \\-  {\bf 1}\{\ell>m\,\text{and}\,t\leq r\} \end{array}$
             & $\begin{array}{l} \min\{\ell-1,m\} \\ + \min\{t,m\} \end{array}$ \\
\hline
$P_1^\ell$ & $\min\{\ell-1, m\}$ 
          & $\min\{\ell,m\}$
          & $\begin{array}{l}  \min\{\ell,m\} \\ +  \min\{\ell-1,r\} \end{array}$
          & $\min\{\ell,m\}$ \\
\hline
\end{tabular}
\end{center}
\end{figure}

We denote by {\bf 1} the {\it characteristic function} corresponding to the property 
specified in parantheses.

\medskip
It will help us simplify the presentation in the next section by introducing the
notation
$$B_2^{m,m-1}\;=\;P_2^m\oplus P_0^{m-1}$$
for $m\geq 2$.
We observe that the notation is consistent with the formulae above.

\subsection{How operations change the hom spaces}\label{section-operations}
%------------------------------------------------

Throughout this section, $Y,Z\in\mathcal S_2$ will be objects of the same partition type.
We introduce two matrices, the multiplicity matrix $\delta M=\delta M(Y,Z)$ 
and the hom matrix $\delta H=\delta H(Y,Z)$; in each case the indexing set 
is the set of isomorphism types of indecomposable objects in $\mathcal S_2$.
The matrices are defined as follows:
$$\delta M_X=\mu_X(Z)-\mu_X(Y), \quad\text{and}\quad 
  \delta H_X=[X,Z]_{\mathcal S}-[X,Y]_{\mathcal S},$$
where $[X,Z]_{\mathcal S}=\dim\Hom_{\mathcal S}(X,Z)$ and 
where $\mu_X(Z)$ denotes the number of direct summands of $Z$ that are isomorphic to $X$.

\medskip
We visualize the matrices by indicating the value at $X\in\ind\mathcal S_2$ 
in the position of $X$ in the Auslander-Reiten quiver for $\mathcal S_2^n$, 
with $n$ large enough. 
We sketch this quiver as follows:  The modules on the top line are 
the $P_2^m$, those on the second line are the $P_0^r$;
the modules in the triangle have type $B_2^{m,r}$. 
The modules $P_1^r$ are repeated twice, on the diagonal at the left
and on the antidiagonal at the right.

$$
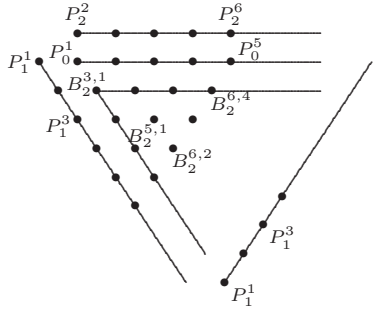
\beginpicture\setcoordinatesystem units <.85cm,.85cm>
\multiput{} at 0 0  4 4 /
\plot -.3 3.9  3.5 3.9 /
\plot -.3 3.45  3.5 3.45 /
\plot 1.7 .45  0 3  3.5 3 /
\plot 2 0  4.3 3.45 / % 2 0  3.7 2.55 /
\plot -.9 3.45  1.4 0 /
\multiput{$\scriptstyle\bullet$} at 2 0  2.3 .45  2.6 .9  2.9  1.35
                                 0 3  .3 2.55  .6 2.1  .9 1.65  .6 3  1.2 3  1.8 3
                                 .9 2.55  1.5 2.55  1.2 2.1  
                                 -.3 3.45  .3 3.45  .9 3.45  1.5 3.45  2.1 3.45 
                                 -.3 3.9  .3 3.9  .9 3.9  1.5 3.9  2.1 3.9 
                                 -.9 3.45  -.6 3  -.3 2.55  0 2.1  .3 1.65  .6 1.2 /
\multiput{$\scriptstyle P_1^1$} at 2.3 -.2  -1.2 3.45 /
\multiput{$\scriptstyle P_1^3$} at 2.9 .7   -.6 2.45 /
\put{$\scriptstyle B_2^{3,1}$} at -.15 3.1
\put{$\scriptstyle B_2^{5,1}$} at .8 2.3
\put{$\scriptstyle B_2^{6,4}$} at 2.1 2.8
\put{$\scriptstyle B_2^{6,2}$} at 1.5 1.9
\put{$\scriptstyle P_0^1$} at -.55 3.6
\put{$\scriptstyle P_2^2$} at -.3 4.2
\put{$\scriptstyle P_0^5$} at 2.4 3.65
\put{$\scriptstyle P_2^6$} at 2.1 4.2
\endpicture
$$

\medskip
\begin{lem}\label{lemma-zeros}
Suppose $Y,Z\in\mathcal S_2$ have the same partition type $(\alpha,\beta,\gamma)$.
\begin{enumerate}
\item The Hom matrix $\delta H(Y,Z)$ has zero entry at each position corresponding
  to a module $P_1^1$, $P_0^m$, $P_2^m$ where $m\in \mathbb N$.
\item Along each diagonal in the Hom matrix, the entries eventually become constant:
  $$\lim_{m\to\infty}\delta H(Y,Z)_{B_2^{m,r}}\;=\; \delta H(Y,Z)_{P_1^m}$$
\end{enumerate}
\end{lem}

\begin{proof}
For the first statement we use the table in the previous section
to verify that the dimensions of
homomorphism spaces are determined by the partition type:
$$\begin{array}{rcl}
  [P_1^1,Y]_{\mathcal S} & = & \bar\alpha_1 \\[.1ex]
  [P_0^m,Y]_{\mathcal S} & = & \bar\beta_1+\cdots+\bar\beta_m \\[.1ex]
  [P_2^m,Y]_{\mathcal S} & = & \bar\alpha_1+\bar\alpha_2+\bar\gamma_1+\cdots \bar\gamma_{m-2}
\end{array}$$

For the second assertion, let $n=\beta_1$ be the nilpotency index of the operator
acting on $Y$.  By comparing the third and the fourth row in the table, we see
that for each $m>n$ the equality $[B_2^{m,r},Y]_{\mathcal S}=|\beta|+[P_1^r,Y]_{\mathcal S}$
holds where $|\beta|=\sum_i\beta_i$.
\end{proof}

\medskip
We determine how the matrices change for each operation on the arc diagram. 

%
%  Case (A)
%

\medskip\noindent{\bf (A)} 
Suppose $Z$ is obtained from $Y$ by a transformation
$$
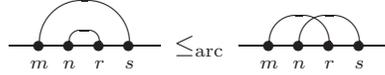
\begin{picture}(20,8)(0,3)
\put(0,4){\line(1,0){20}}
\multiput(3,3)(4,0)4{$\bullet$}
\put(10,4){\oval(4,4)[t]}
\put(10,4){\oval(12,12)[t]}
\put(4,1){\makebox[0mm]{$\scriptstyle m$}}
\put(8,1){\makebox[0mm]{$\scriptstyle n$}}
\put(12,1){\makebox[0mm]{$\scriptstyle r$}}
\put(16,1){\makebox[0mm]{$\scriptstyle s$}}
\end{picture}
\;
\leq_{\rm arc}
\;
\begin{picture}(20,8)(0,3)
\put(0,4){\line(1,0){20}}
\multiput(3,3)(4,0)4{$\bullet$}
\put(8,4){\oval(8,8)[t]}
\put(12,4){\oval(8,8)[t]}
\put(4,1){\makebox[0mm]{$\scriptstyle m$}}
\put(8,1){\makebox[0mm]{$\scriptstyle n$}}
\put(12,1){\makebox[0mm]{$\scriptstyle r$}}
\put(16,1){\makebox[0mm]{$\scriptstyle s$}}
\end{picture}
$$
where $m>n>r>s$ and $n>r+1$.
Recall that the two arcs in the diagram for $Y$ represent 
direct summands $B_2^{m,s}$ and $B_2^{n,r}$, which are replaced by
the summands $B_2^{m,r}$ and $B_2^{n,s}$ in $Z$ that give rise
to the corresponding arcs in the diagram for $Z$.

\smallskip
Thus, the multiplicity matrix is as follows.
$$
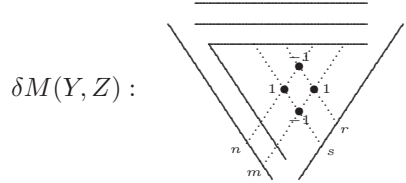
\beginpicture\setcoordinatesystem units <.6cm,.6cm>
\multiput{} at 0 0  4 4 /
\put{$\delta M(Y,Z):$} at -3 2
\plot -.3 3.9  3.5 3.9 /
\plot -.3 3.45  3.5 3.45 /
\plot 1.7 .45  0 3  3.5 3 /
\plot 2 0  4.3 3.45 /
\plot 1.4 0  -.9 3.45 /
\setdots<2pt>
\plot .84 .8  2.33 3 /
\plot 1.2 .3  3 3 /
\plot 2.5 .75  1 3 /
\plot 2.84 1.25  1.67 3 /
\put{$\ssize n$} at .6 .65
\put{$\ssize m$} at 1 .15
\put{$\ssize s$} at 2.7 .6
\put{$\ssize r$} at 3 1.1
\multiput{$\scriptstyle\bullet$} at 2 1.5  1.67 2  2.33 2  2 2.5 /
\multiput{$\ssize 1$} at 1.4 2  2.6 2 /
\multiput{$\ssize -1$} at 2 1.3  2 2.7 /
\endpicture
$$

Note that the marked points correspond to a short exact sequence
$$0\longrightarrow B_2^{n,s}\longrightarrow B_2^{m,s}\oplus B_2^{n,r}\longrightarrow B_2^{m,r}\longrightarrow 0$$
which serves as a witness for the implication 
$$Y\arcleq Z\;\Longrightarrow\;Y\extleq Z.$$

\medskip
Next we determine the Hom matrix $\delta H=\delta H(B_2^{m,s}\oplus B_2^{n,r},B_2^{n,s}\oplus  B_2^{m,r})$: 
By Lemma~\ref{lemma-zeros}, $\delta H_{P_0^\ell}=0=\delta H_{P_2^\ell}$.  
We compute using the table in Section~\ref{section-dim-hom}: $\delta H_{P_1^\ell}=0$ and
$$\delta H_{B_2^{\ell,t}}={\bf 1}\{n<\ell\leq m\,\text{and}\,s<t\leq r\}.$$
Thus, the only indecomposables $X\in\mathcal S_2$
for which $\delta H_X\neq 0$ are the $B_2^{\ell,t}$ where
$n<\ell\leq m$ and $s<t\leq r$.  For each such module $X$ we have $\delta H_X=1$.
They lie in the shaded region in the diagram below.

\smallskip
$$
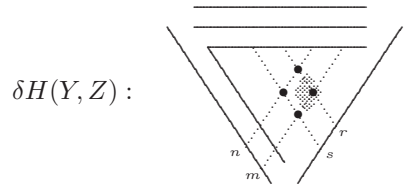
\beginpicture\setcoordinatesystem units <.6cm,.6cm>
\multiput{} at 0 0  4 4 /
\put{$\delta H(Y,Z):$} at -3 2
\plot -.3 3.9  3.5 3.9 /
\plot -.3 3.45  3.5 3.45 /
\plot 1.7 .45  0 3  3.5 3 /
\plot 2 0  4.3 3.45 /
\plot 1.4 0  -.9 3.45 /
\setdots<2pt>
\plot .84 .8  2.33 3 /
\plot 1.2 .3  3 3 /
\plot 2.5 .75  1 3 /
\plot 2.84 1.25  1.67 3 /
\put{$\ssize n$} at .6 .65
\put{$\ssize m$} at 1 .15
\put{$\ssize s$} at 2.7 .6
\put{$\ssize r$} at 3 1.1
\multiput{$\scriptstyle\bullet$} at 2 1.5  1.67 2  2.33 2  2 2.5 /
\setshadegrid span <.3mm>
\vshade 1.87 2 2 <,z,,>
        2.2 1.5 2.5  <z,z,,>
        2.53 2 2  /
\endpicture
$$

%
%  Case (A')
%
\medskip\noindent{\bf (A')} 
In this case  $Z$ is obtained from $Y$ by an arc operation of type
$$
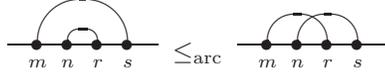
\begin{picture}(20,8)(0,2)
\put(0,4){\line(1,0){20}}
\multiput(3,3)(4,0)4{$\bullet$}
\put(10,4){\oval(4,4)[t]}
\put(10,4){\oval(12,12)[t]}
\put(4,1){\makebox[0mm]{$\scriptstyle m$}}
\put(8,1){\makebox[0mm]{$\scriptstyle n$}}
\put(12,1){\makebox[0mm]{$\scriptstyle r$}}
\put(16,1){\makebox[0mm]{$\scriptstyle s$}}
\end{picture}
\;
\leq_{\rm arc}
\;
\begin{picture}(20,8)(0,2)
\put(0,4){\line(1,0){20}}
\multiput(3,3)(4,0)4{$\bullet$}
\put(8,4){\oval(8,8)[t]}
\put(12,4){\oval(8,8)[t]}
\put(4,1){\makebox[0mm]{$\scriptstyle m$}}
\put(8,1){\makebox[0mm]{$\scriptstyle n$}}
\put(12,1){\makebox[0mm]{$\scriptstyle r$}}
\put(16,1){\makebox[0mm]{$\scriptstyle s$}}
\end{picture}
$$
where $m>n>r>s$, as in {\bf (A)}, but now $n=r+1$.
Here the two arcs in the diagram for $Y$ represent 
direct summands $B_2^{m,s}$ and (corresponding to the arc from $n$ to $r$,)
$P_2^{n}$ and $P_0^r$; they are replaced by
the summands $B_2^{m,r}$ and $B_2^{n,s}$ in $Z$.

\medskip
In this case, the short exact sequence demonstrates the implication $Y\arcleq Z\Rightarrow Y\extleq Z$:
$$0\longrightarrow B_2^{n,s}\longrightarrow B_2^{m,s}\oplus P_2^n \oplus P_0^r\longrightarrow B_2^{m,r}\longrightarrow 0$$

\smallskip
We picture the corresponding multiplicity and Hom matrices. 
In view of the observation after the table in Section~\ref{section-dim-hom},
the computations for the Hom matrix in {\bf (A)} are still valid in this case:
$\delta H_X=1$ for $X=B_2^{\ell,t}$ where $n<\ell\leq m$, $s<t\leq r$.

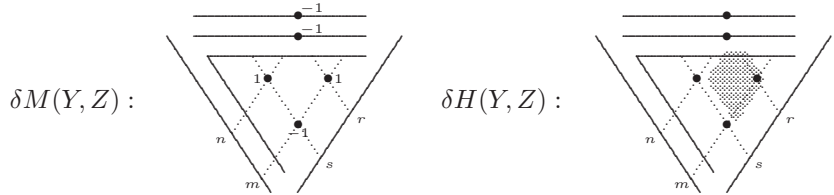
$$
\beginpicture\setcoordinatesystem units <.6cm,.6cm>
\multiput{} at 0 0  4 4 /
\put{$\delta M(Y,Z):$} at -3 2
\plot -.3 3.9  3.5 3.9 /
\plot -.3 3.45  3.5 3.45 /
\plot 1.7 .45  0 3  3.5 3 /
\plot 2 0  4.3 3.45 /
\plot 1.4 0  -.9 3.45 /
\setdots<2pt>
\plot .54 1.3  1.67 3 /
\plot 1.2 .3  3 3 /
\plot 2.5 .75  1 3 /
\plot 2.33 3  3.16 1.75 /
\put{$\ssize n$} at .3 1.15
\put{$\ssize m$} at 1 .15
\put{$\ssize s$} at 2.7 .6
\put{$\ssize r$} at 3.4 1.6
\multiput{$\scriptstyle\bullet$} at 2 1.5  1.33 2.5  2.67 2.5  2 3.45  2 3.9 /
\multiput{$\ssize 1$} at 1.1 2.5  2.9 2.5 /
\multiput{$\ssize -1$} at 2 1.3  2.3 3.6  2.3 4.05 /
\endpicture
\qquad
\beginpicture\setcoordinatesystem units <.6cm,.6cm>
\multiput{} at 0 0  4 4 /
\put{$\delta H(Y,Z):$} at -3 2
\plot -.3 3.9  3.5 3.9 /
\plot -.3 3.45  3.5 3.45 /
\plot 1.7 .45  0 3  3.5 3 /
\plot 2 0  4.3 3.45 /
\plot 1.4 0  -.9 3.45 /
\setdots<2pt>
\plot .54 1.3  1.67 3 /
\plot 1.2 .3  3 3 /
\plot 2.5 .75  1 3 /
\plot 2.33 3  3.16 1.75 /
\put{$\ssize n$} at .3 1.15
\put{$\ssize m$} at 1 .15
\put{$\ssize s$} at 2.7 .6
\put{$\ssize r$} at 3.4 1.6
\multiput{$\scriptstyle\bullet$} at 2 1.5  1.33 2.5  2.67 2.5  2 3.45  2 3.9 /
\setshadegrid span <.3mm>
\vshade 1.53 2.5 2.5  <,z,,>
        1.97 1.85 3.15  <z,z,,>
        2.2 1.5 3.15   <z,z,,>
        2.43 1.85 3.15  <z,z,,>
        2.87 2.5 2.5 /
\endpicture
$$
%
%     Case (B)
%

\medskip\noindent{\bf (B)}
Next we deal with the transformation
$$
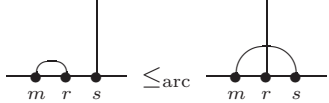
\begin{picture}(16,8)(0,3)
\put(0,4){\line(1,0){16}}
\multiput(3,3)(4,0)3{$\bullet$}
\put(6,4){\oval(4,4)[t]}
\put(12,4){\line(0,1){10}}
\put(4,1){\makebox[0mm]{$\scriptstyle m$}}
\put(8,1){\makebox[0mm]{$\scriptstyle r$}}
\put(12,1){\makebox[0mm]{$\scriptstyle s$}}
\end{picture}
\;
\leq_{\rm arc}
\;
\begin{picture}(16,8)(0,3)
\put(0,4){\line(1,0){16}}
\multiput(3,3)(4,0)3{$\bullet$}
\put(8,4){\oval(8,8)[t]}
\put(8,4){\line(0,1){10}}
\put(4,1){\makebox[0mm]{$\scriptstyle m$}}
\put(8,1){\makebox[0mm]{$\scriptstyle r$}}
\put(12,1){\makebox[0mm]{$\scriptstyle s$}}
\end{picture}
$$
where $m>r>s$.  We consider the two cases where $m-1>r$ and where $m-1=r$ simultaneously
using the notation $B_2^{m,m-1}=P_2^m\oplus P_0^{m-1}$.  
Thus, $Y$ has summands $B_2^{m,r}$ and $P_1^s$
which are replaced in the transformation by summands $B_2^{m,s}$ and $P_1^r$
for $Z$. 

\medskip
The following short exact sequence demonstrates that $Y\arcleq Z$ implies $Y\extleq Z$:
$$0\longrightarrow B_2^{m,s}\longrightarrow B_2^{m,r}\oplus P_1^s\longrightarrow P_1^r\longrightarrow 0$$

\smallskip
Here are the corresponding multiplicity and Hom-matrices.
$$
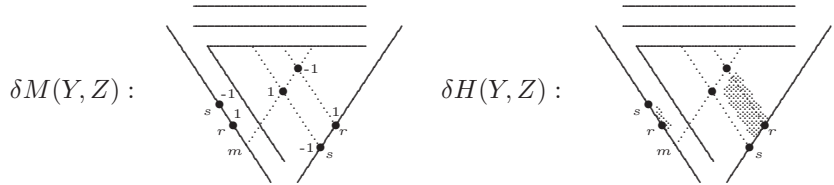
\beginpicture\setcoordinatesystem units <.6cm,.6cm>
\multiput{} at 0 0  4 4 /
\put{$\delta M(Y,Z):$} at -3 2
\plot -.3 3.9  3.5 3.9 /
\plot -.3 3.45  3.5 3.45 /
\plot 1.7 .45  0 3  3.5 3 /
\plot 2 0  4.3 3.45 /
\plot 1.4 0  -.9 3.45 /
\setdots<2pt>
\plot .86 .8  2.33 3 /
\plot 2.5 .75  1 3 /
\plot 2.84 1.25  1.67 3 /
\put{$\ssize s$} at 0 1.55
\put{$\ssize r$} at .3 1.1
\put{$\ssize m$} at .6 .65
\put{$\ssize s$} at 2.7 .6
\put{$\ssize r$} at 3.0 1.1
\multiput{$\scriptstyle\bullet$} at 1.67 2  2 2.5  2.5 .75  2.84 1.25  .56 1.25  .26 1.7 /
\multiput{$\ssize 1$} at 1.4 2  2.84 1.55 .66 1.5 /
\multiput{$\ssize \text-1$} at 2.3 2.5  2.2 .75  .46 1.95  /
\endpicture
\qquad
\beginpicture\setcoordinatesystem units <.6cm,.6cm>
\multiput{} at 0 0  4 4 /
\put{$\delta H(Y,Z):$} at -3 2
\plot -.3 3.9  3.5 3.9 /
\plot -.3 3.45  3.5 3.45 /
\plot 1.7 .45  0 3  3.5 3 /
\plot 2 0  4.3 3.45 /
\plot 1.4 0  -.9 3.45 /
\setdots<2pt>
\plot .86 .8  2.33 3 /
\plot 2.5 .75  1 3 /
\plot 2.84 1.25  1.67 3 /
\put{$\ssize s$} at 0 1.55
\put{$\ssize r$} at .3 1.1
\put{$\ssize m$} at .6 .65
\put{$\ssize s$} at 2.7 .6
\put{$\ssize r$} at 3.0 1.1
\multiput{$\scriptstyle\bullet$} at 1.67 2  2 2.5  2.5 .75  2.84 1.25  .56 1.25  .26 1.7 /
%\multiput{$\ssize 1$} at 1.4 2  2.84 1.55 .66 1.5 /
%\multiput{$\ssize \text-1$} at 2.3 2.5  2.2 .75  .46 1.95  /
\setshadegrid span <.3mm>
\vshade 1.87 2 2 <,z,,>
        2.2 1.5 2.5  <z,z,,>
        2.6  .9  1.9  <z,z,,>
        2.94 1.4 1.4 /
\vshade .36 1.55 1.55 <,z,,>
        .46 1.4 1.7 <z,z,,>
        .66 1.1 1.4 <z,z,,>
        .76 1.25 1.25 /
\endpicture
$$
Here $\delta H_X=1$ for modules $X$ of type $B_2^{\ell,t}$ where $m<\ell$ and $s<t\leq r$,
and for modules of type $P_1^t$ where $s<t\leq r$.
%
%  Case (C)
%

\medskip\noindent{\bf (C)} 
Here $Z$ is obtained from $Y$ by a transformation of type
$$
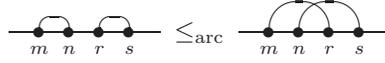
\begin{picture}(20,8)(0,3)
\put(0,4){\line(1,0){20}}
\multiput(3,3)(4,0)4{$\bullet$}
\put(6,4){\oval(4,4)[t]}
\put(14,4){\oval(4,4)[t]}
\put(4,1){\makebox[0mm]{$\scriptstyle m$}}
\put(8,1){\makebox[0mm]{$\scriptstyle n$}}
\put(12,1){\makebox[0mm]{$\scriptstyle r$}}
\put(16,1){\makebox[0mm]{$\scriptstyle s$}}
\end{picture}
\;
\leq_{\rm arc}
\;
\begin{picture}(20,8)(0,3)
\put(0,4){\line(1,0){20}}
\multiput(3,3)(4,0)4{$\bullet$}
\put(8,4){\oval(8,8)[t]}
\put(12,4){\oval(8,8)[t]}
\put(4,1){\makebox[0mm]{$\scriptstyle m$}}
\put(8,1){\makebox[0mm]{$\scriptstyle n$}}
\put(12,1){\makebox[0mm]{$\scriptstyle r$}}
\put(16,1){\makebox[0mm]{$\scriptstyle s$}}
\end{picture}
$$
where $m>n>r>s$.
Recall that the two arcs in the diagram for $Y$ represent 
direct summands $B_2^{m,n}$ and $B_2^{r,s}$, which are replaced by
the summands $B_2^{m,r}$ and $B_2^{n,s}$ in $Z$ that give rise
to the corresponding arcs in the diagram for $Z$.
Note that we are really dealing with 4 cases depending on 
whether $m=n+1$ (in which case $B_2^{m,n}=P_2^m\oplus P_0^n$ 
is decomposable) or $m>n+1$, and whether $r=s+1$ (in which case
$B_2^{r,s}=P_2^r\oplus P_0^s$ is decomposable) or $r>s+1$.
We picture the matrices only for the case where $m>n+1$
and $r>s+1$; the modifications for the remaining cases are
as in {\bf (A)} and {\bf (A')}.

\smallskip
The multiplicity matrix is as follows.
$$
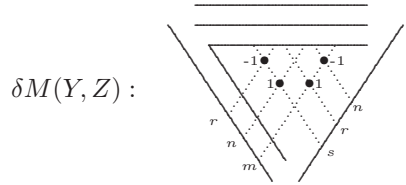
\beginpicture\setcoordinatesystem units <.6cm,.6cm>
\multiput{} at 0 0  4 4 /
\put{$\delta M(Y,Z):$} at -3 2
\plot -.3 3.9  3.5 3.9 /
\plot -.3 3.45  3.5 3.45 /
\plot 1.7 .45  0 3  3.5 3 /
\plot 2 0  4.3 3.45 /
\plot 1.4 0  -.9 3.45 /
\setdots<2pt>
\plot .43 1.45  1.47 3 /
\plot .77 .95  2.13 3 /
\plot 1.1 .45 2.8 3 /
\plot 2.5 .75  1 3 /
\plot 2.83 1.25  1.67 3 /
\plot 3.17 1.75  2.33 3 /
\put{$\ssize r$} at .1 1.3
\put{$\ssize n$} at .5 .8
\put{$\ssize m$} at .9 .3
\put{$\ssize s$} at 2.7 .6
\put{$\ssize r$} at 3 1.1
\put{$\ssize n$} at 3.3 1.6
\multiput{$\scriptstyle\bullet$} at 1.57 2.15  2.23 2.15  1.23 2.65  2.55 2.65 /
\multiput{$\ssize 1$} at 1.37 2.15  2.43 2.15 /
\multiput{$\ssize \text-1$} at .93 2.65  2.85 2.65 /
\endpicture
$$

The marked points correspond to a short exact sequence
$$0\longrightarrow B_2^{m,r}\longrightarrow B_2^{m,n}\oplus B_2^{r,s}\longrightarrow B_2^{n,s}\longrightarrow 0$$
which confirms the  implication $Y\arcleq Z\Rightarrow Y\extleq Z.$
For the Hom matrix $\delta H=\delta H(B_2^{m,n}\oplus B_2^{r,s},B_2^{n,s}\oplus  B_2^{m,r})$ we use
the table in Section~\ref{section-dim-hom}: 
$$\begin{array}{rcl} \delta H_{P_1^\ell} & = & {\bf 1}\{r<\ell\leq n\}\\
                     \delta H_{B_2^{\ell,t}} & = & {\bf 1}\{r<\ell\leq n\,\text{and}\, t\leq s\}+{\bf 1}\{m<\ell\,\text{and}\, r<t\leq n\}.
\end{array}
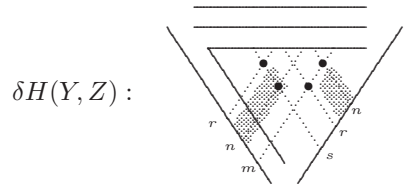
$$

$$
\beginpicture\setcoordinatesystem units <.6cm,.6cm>
\multiput{} at 0 0  4 4 /
\put{$\delta H(Y,Z):$} at -3 2
\plot -.3 3.9  3.5 3.9 /
\plot -.3 3.45  3.5 3.45 /
\plot 1.7 .45  0 3  3.5 3 /
\plot 2 0  4.3 3.45 /
\plot 1.4 0  -.9 3.45 /
\setdots<2pt>
\plot .43 1.45  1.47 3 /
\plot .77 .95  2.13 3 /
\plot 1.1 .45 2.8 3 /
\plot 2.5 .75  1 3 /
\plot 2.83 1.25  1.67 3 /
\plot 3.17 1.75  2.33 3 /
\put{$\ssize r$} at .1 1.3
\put{$\ssize n$} at .5 .8
\put{$\ssize m$} at .9 .3
\put{$\ssize s$} at 2.7 .6
\put{$\ssize r$} at 3 1.1
\put{$\ssize n$} at 3.3 1.6
\multiput{$\scriptstyle\bullet$} at 1.57 2.15  2.23 2.15  1.23 2.65  2.55 2.65 /
%\multiput{$\ssize 1$} at 1.37 2.15  2.43 2.15 /
%\multiput{$\ssize \text-1$} at .93 2.65  2.85 2.65 /
\setshadegrid span <.3mm>
\vshade .53 1.3 1.3 <,z,,>
        .87 .8  1.8  <z,z,,>
        1.43  1.65 2.65 <z,z,,>
        1.77 2.15 2.15 /
\vshade 2.43 2.15 2.15 <,z,,>
        2.75 1.65 2.65 <z,z,,>
        2.93 1.4 2.4  <z,z,,>
        3.27 1.9 1.9  /
\endpicture
$$

%
%  Case (D) 
%
\medskip\noindent{\bf (D)}
Finally, we deal with the transformation
$$
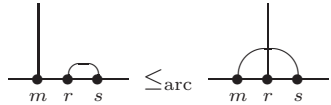
\begin{picture}(16,8)(0,3)
\put(0,4){\line(1,0){16}}
\multiput(3,3)(4,0)3{$\bullet$}
\put(10,4){\oval(4,4)[t]}
\put(4,4){\line(0,1){10}}
\put(4,1){\makebox[0mm]{$\scriptstyle m$}}
\put(8,1){\makebox[0mm]{$\scriptstyle r$}}
\put(12,1){\makebox[0mm]{$\scriptstyle s$}}
\end{picture}
\;
\leq_{\rm arc}
\;
\begin{picture}(16,8)(0,3)
\put(0,4){\line(1,0){16}}
\multiput(3,3)(4,0)3{$\bullet$}
\put(8,4){\oval(8,8)[t]}
\put(8,4){\line(0,1){10}}
\put(4,1){\makebox[0mm]{$\scriptstyle m$}}
\put(8,1){\makebox[0mm]{$\scriptstyle r$}}
\put(12,1){\makebox[0mm]{$\scriptstyle s$}}
\end{picture}
$$
where $m>r>s$.  Thus, $Y$ has summands $B_2^{r,s}$ and $P_1^m$
which are replaced in the transformation by summands $B_2^{m,s}$ and $P_1^r$
for $Z$. 
The following short exact sequence demonstrates that $Y\arcleq Z$ implies $Y\extleq Z$:
$$0\longrightarrow P_1^r\longrightarrow P_1^m\oplus B_2^{r,s}\longrightarrow B_2^{m,s}\longrightarrow 0$$

\smallskip
Here are the corresponding multiplicity and Hom-matrices.
$$
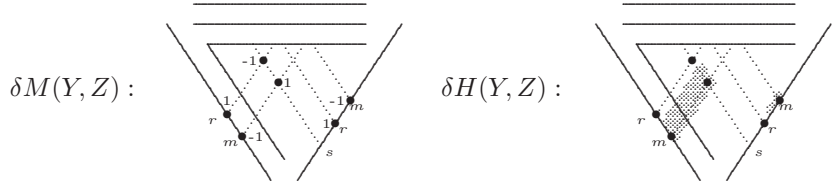
\beginpicture\setcoordinatesystem units <.6cm,.6cm>
\multiput{} at 0 0  4 4 /
\put{$\delta M(Y,Z):$} at -3 2
\plot -.3 3.9  3.5 3.9 /
\plot -.3 3.45  3.5 3.45 /
\plot 1.7 .45  0 3  3.5 3 /
\plot 2 0  4.3 3.45 /
\plot 1.4 0  -.9 3.45 /
\setdots<2pt>
\plot .43 1.45  1.47 3 /
\plot .77 .95  2.13 3 /
%\plot 1.1 .45 2.8 3 /
\plot 2.5 .75  1 3 /
\plot 2.83 1.25  1.67 3 /
\plot 3.17 1.75  2.33 3 /
\put{$\ssize r$} at .1 1.3
\put{$\ssize m$} at .5 .8
%\put{$\ssize m$} at .9 .3
\put{$\ssize s$} at 2.7 .6
\put{$\ssize r$} at 3 1.1
\put{$\ssize m$} at 3.3 1.6
\multiput{$\scriptstyle\bullet$} at 1.57 2.15  1.23 2.65  .43 1.45  .77 .95  2.83 1.25  3.17 1.75 /
\multiput{$\ssize 1$} at 1.77 2.15  .43 1.75  2.63 1.25  /
\multiput{$\ssize \text-1$} at .93 2.65  1.07 .95  2.87 1.75 /
\endpicture
\qquad
\beginpicture\setcoordinatesystem units <.6cm,.6cm>
\multiput{} at 0 0  4 4 /
\put{$\delta H(Y,Z):$} at -3 2
\plot -.3 3.9  3.5 3.9 /
\plot -.3 3.45  3.5 3.45 /
\plot 1.7 .45  0 3  3.5 3 /
\plot 2 0  4.3 3.45 /
\plot 1.4 0  -.9 3.45 /
\setdots<2pt>
\plot .43 1.45  1.47 3 /
\plot .77 .95  2.13 3 /
%\plot 1.1 .45 2.8 3 /
\plot 2.5 .75  1 3 /
\plot 2.83 1.25  1.67 3 /
\plot 3.17 1.75  2.33 3 /
\put{$\ssize r$} at .1 1.3
\put{$\ssize m$} at .5 .8
%\put{$\ssize m$} at .9 .3
\put{$\ssize s$} at 2.7 .6
\put{$\ssize r$} at 3 1.1
\put{$\ssize m$} at 3.3 1.6
\multiput{$\scriptstyle\bullet$} at 1.57 2.15  1.23 2.65  .43 1.45  .77 .95  2.83 1.25  3.17 1.75 /
%\multiput{$\ssize 1$} at 1.77 2.15  .43 1.75  2.63 1.25  /
%\multiput{$\ssize \text-1$} at .93 2.65  1.07 .95  2.87 1.75 /
\setshadegrid span <.3mm>
\vshade .53 1.3 1.3 <,z,,>
        .87 .8  1.8  <z,z,,>
        1.43  1.65 2.65 <z,z,,>
        1.77 2.15 2.15 /
\vshade 2.83 1.55 1.55 <,z,,>
        2.93 1.4 1.7   <z,z,,>
        3.17 1.75 2.05 <z,z,,>
        3.27 1.9 1.9 /
\endpicture
$$

We proved the following fact.

\begin{lem}\label{lem-tab-ext}
  Let $k$ be an~arbitrary field and let
  $Y,Z\in\mathcal{S}_2(k)$ have the same partition type 
  $(\alpha,\beta,\gamma)$.
  If $Y\leq_{\rm arc} Z$, then $Y\leq_{\rm ext} Z$.\qed
\end{lem} 

Note that the Hom matrix determines the multiplicity matrix uniquely.
Namely, let $A$ be a noninjective indecomposable object with
Auslander-Reiten sequence $0\to A\to \bigoplus B_i\to C\to 0$ 
in $\mathcal S_2^n$.  Then the multiplicity of $A$ as a direct summand
of $Y$ is given by the contravariant defect
$\mu_A(Y)=[A,Y]+[C,Y]-\sum_i [B_i,Y]$.

\smallskip
The following consequence, which is technical but easy to show, will be used 
in the proof of Proposition~\ref{proposition-key-lemma}.

  \begin{lem}\label{lem-hom-yields-mult}
    Suppose that both $Y,Z\in\mathcal S_2^n$ have partition type $(\alpha,\beta,\gamma)$
    with $\alpha_1\leq 2$. 
    \begin{enumerate}
    \item
      The nonzero part of the Hom matrix $\delta H(Y,Z)$ 
      is contained in the union of the $\tau$-orbits for $X=B_2^{3,1},\ldots, B_2^{n,1}$;
      they form a Moebius band, i.e.\ a quiver of type $\mathbb Z \mathbb A_{n-2}$ 
      with suitable identifications.
    \item
      For each non-injective $A\in\ind\,\mathcal S_2^n$, the entry $\delta M_A=\delta M(Y,Z)_A$ 
      in the multiplicity matrix 
      can be read off from the the restriction of the Hom matrix to the Moebius band
      by a formula of type
      $$\delta M_A=\beta_A+\beta_C-\sum_i\beta_{B_i}$$
      where $0\to A\to \bigoplus_i B_i\to C\to 0$ is the Auslander-Reiten sequence starting at $A$ and 
      $$\beta_X=\left\{\begin{array}{ll}\delta H(Y,Z)_X & \text{if $X$ is in the Moebius band}\\
      0               & \text{otherwise.}\end{array}\right.$$
    \end{enumerate}
  \end{lem}

  \begin{proof}
    The first statement follows from Lemma~\ref{lemma-zeros}, the second 
    from the contravariant defect formula above.
  \end{proof}

\subsection{An example}\label{section-example}
%----------------------

In this section we present the example from the introduction to
Chapter~\ref{chapter-hom-tab}  in detail.  The arc diagrams
$$
\begin{picture}(32,15)(0,3)
\put(-8,8){$Z:$}
\put(0,4){\line(1,0){32}}
\multiput(3,3)(4,0)7{$\bullet$}
\put(14,4){\oval(20,20)[t]}
\put(20,4){\oval(16,16)[t]}
\put(14,4){\oval(12,12)[t]}
\put(16,4){\line(0,1){15}}
\put(4,1){\makebox[0mm]{$\scriptstyle 7$}}
\put(8,1){\makebox[0mm]{$\scriptstyle 6$}}
\put(12,1){\makebox[0mm]{$\scriptstyle 5$}}
\put(16,1){\makebox[0mm]{$\scriptstyle 4$}}
\put(20,1){\makebox[0mm]{$\scriptstyle 3$}}
\put(24,1){\makebox[0mm]{$\scriptstyle 2$}}
\put(28,1){\makebox[0mm]{$\scriptstyle 1$}}
\end{picture}
\qquad\qquad\qquad
\begin{picture}(32,15)(0,3)
\put(-8,8){$Y:$}
\put(0,4){\line(1,0){32}}
\multiput(3,3)(4,0)7{$\bullet$}
\put(12,4){\oval(16,16)[t]}
\put(16,4){\oval(16,16)[t]}
\put(14,4){\oval(4,4)[t]}
\put(28,4){\line(0,1){15}}
\put(4,1){\makebox[0mm]{$\scriptstyle 7$}}
\put(8,1){\makebox[0mm]{$\scriptstyle 6$}}
\put(12,1){\makebox[0mm]{$\scriptstyle 5$}}
\put(16,1){\makebox[0mm]{$\scriptstyle 4$}}
\put(20,1){\makebox[0mm]{$\scriptstyle 3$}}
\put(24,1){\makebox[0mm]{$\scriptstyle 2$}}
\put(28,1){\makebox[0mm]{$\scriptstyle 1$}}
\end{picture}
$$
represent the modules 
$Y=B_2^{7,3}\oplus B_2^{6,2}\oplus P_2^5\oplus P_0^4\oplus P_1^1$
and $Z=B_2^{7,2}\oplus B_2^{6,3}\oplus B_2^{5,1}\oplus P_1^4$.

\medskip
We first compute the Hom-matrix to verify that $Y\leq_{\rm hom}Z$.
In this matrix, the entries on the first diagonal represent the numbers
$\delta H_{B_2^{3,1}},\delta H_{B_2^{4,1}},\ldots$. 
By Lemma~\ref{lemma-zeros}, there are also zeros along the
two top rows, which we indicate by solid lines.

$$
\beginpicture\setcoordinatesystem units <.4cm,.4cm>
\multiput{} at 0 0  4 4 /
\put{$\delta H(Y,Z):$} at -3 2
\plot -.3 3.9  6.5 3.9 /
\plot -.3 4.1  6.5 4.1  /
\multiput{$\scriptstyle 0$} at 0 3.5  1 3.5  2 3.5  4 3.5  5 3.5  
                               .5 3  1.5 3  4.5 3  6.5 3 
                               1 2.5  3 2.5  6 2.5
                               1.5 2  5.5 2
                               2 1.5 
                               2.5 1
                               3.5 0 /
\multiput{$\scriptstyle 1$} at 3 3.5  
                               2.5 3  3.5 3
                               2 2.5  4 2.5
                               2.5 2  3.5 2
                               3 1.5  5 1.5
                                      4.5 1
                                      4 .5 /
\multiput{$\cdot$} at        2.8 .7  3 .5  3.2 .3
                             3.3 1.2  3.5 1  3.7 .8 
                             3.8 1.7  4 1.5  4.2 1.3
                             4.3 2.2  4.5 2  4.7 1.8 
                             4.8 2.7  5 2.5  5.2 2.3
                             6.8 3.3  7 3.5  7.2 3.7 /
\endpicture
$$

We ask:  Is it true that $Y\leq_{\rm arc}Z$?  If so, which operations 
transform the arc diagram of $Z$ into the arc diagram of $Y$?

\smallskip
First we consider again 
the single operation of type {\bf (B')} on $Z$, replacing the arc from 5 to 1, 
i.e.\ the bipicket $B_2^{5,1}$ and the
pole at 4, i.e.\ the picket $P_1^4$ by the arc given by the pickets 
$P_2^5$, $P_0^4$ and the pole at 1 given by the picket $P_1^1$, to obtain the
module $\tilde Z$ with the following arc diagram.

$$
\begin{picture}(32,15)(0,3)
\put(-8,8){$\tilde Z:$}
\put(0,4){\line(1,0){32}}
\multiput(3,3)(4,0)7{$\bullet$}
\put(14,4){\oval(20,20)[t]}
\put(14,4){\oval(12,12)[t]}
\put(14,4){\oval(4,4)[t]}
\put(28,4){\line(0,1){15}}
\put(4,1){\makebox[0mm]{$\scriptstyle 7$}}
\put(8,1){\makebox[0mm]{$\scriptstyle 6$}}
\put(12,1){\makebox[0mm]{$\scriptstyle 5$}}
\put(16,1){\makebox[0mm]{$\scriptstyle 4$}}
\put(20,1){\makebox[0mm]{$\scriptstyle 3$}}
\put(24,1){\makebox[0mm]{$\scriptstyle 2$}}
\put(28,1){\makebox[0mm]{$\scriptstyle 1$}}
\end{picture}
$$

This single misstep yields $\tilde Z\leq_{\rm hom}Y$ since, as 
we have seen in Section~\ref{section-operations} part (B'),
all entries in the big parallelogram in the Hom-matrix have been reduced by 1.

$$
\beginpicture\setcoordinatesystem units <.4cm,.4cm>
\multiput{} at 0 0  4 4 /
\put{$\delta H(Y,\tilde Z):$} at -3 2
\plot -.3 3.9  6.5 3.9 /
\plot -.3 4.1  6.5 4.1  /
\multiput{$\scriptstyle 0$} at 0 3.5  1 3.5  2 3.5  4 3.5  5 3.5  
                               .5 3  1.5 3  4.5 3  6.5 3 
                               1 2.5        6 2.5
                               1.5 2  5.5 2
                               2 1.5 
                               2.5 1
                               3.5 0 /
\multiput{$\scriptstyle 0$} at 3 3.5  
                               2.5 3  3.5 3
                               2 2.5  4 2.5
                               2.5 2  3.5 2
                               3 1.5  5 1.5
                                      4.5 1
                                      4 .5 /
\put{$\scriptstyle -1$} at 3 2.5
\setdots<2pt>
\plot 1.5 2.5  3 4  5.5 1.5  4 0  1.5 2.5 /  
\multiput{$\cdot$} at        2.8 .7  3 .5  3.2 .3
                             3.3 1.2  3.5 1  3.7 .8 
                             3.8 1.7  4 1.5  4.2 1.3
                             4.3 2.2  4.5 2  4.7 1.8 
                             4.8 2.7  5 2.5  5.2 2.3
                             6.8 3.3  7 3.5  7.2 3.7 /
\endpicture
$$

So we start over.  Now we focus on the smaller parallelogram in the Hom-matrix,
as indicated.  

$$
\beginpicture\setcoordinatesystem units <.4cm,.4cm>
\multiput{} at 0 0  4 4 /
\put{$\delta H(Y,Z):$} at -3 2
\plot -.3 3.9  6.5 3.9 /
\plot -.3 4.1  6.5 4.1  /
\multiput{$\scriptstyle 0$} at 0 3.5  1 3.5  2 3.5  4 3.5  5 3.5  
                               .5 3  1.5 3  4.5 3  6.5 3 
                               1 2.5  3 2.5  6 2.5
                               1.5 2  5.5 2
                               2 1.5 
                               2.5 1
                               3.5 0 /
\multiput{$\scriptstyle 1$} at 3 3.5  
                               2.5 3  3.5 3
                               2 2.5  4 2.5
                               2.5 2  3.5 2
                               3 1.5  5 1.5
                                      4.5 1
                                      4 .5 /
\multiput{$\cdot$} at        2.8 .7  3 .5  3.2 .3
                             3.3 1.2  3.5 1  3.7 .8 
                             3.8 1.7  4 1.5  4.2 1.3
                             4.3 2.2  4.5 2  4.7 1.8 
                             4.8 2.7  5 2.5  5.2 2.3
                             6.8 3.3  7 3.5  7.2 3.7 /
\setdots<2pt>
\plot 2.5 3.5  1.5 2.5  2 2  3 3  2.5 3.5 /
\endpicture
$$

There are two reasons for considering this parallelogram:
(1) On the left of the left corner of the parallelogram, and in the right corner, 
there are positive entries in the multiplicity
diagram, see Section~\ref{section-operations}.  Those entries 
indicate the arcs and poles involved in the operation which 
produces $Z'$ from $Z$.

\smallskip
(2) Unlike the big parallelogram, there are only positive
entries in the smaller parallelogram.  This will make sure that
$Y\leq_{\rm hom} Z'$ holds. 

\medskip
The corresponding operation of type {\bf (A)} 
replaces the bipickets $B_2^{5,1}$ and $B_2^{6,3}$ in $Z$ (representing arcs 
from 5 to 1 and from 6 to 3, respectively)
by bipickets $B_2^{6,1}$ and $B_2^{5,3}$ in $Z'$, so $Z'$ is given by
the following arc diagram.  We also indicate the new Hom-matrix.

$$
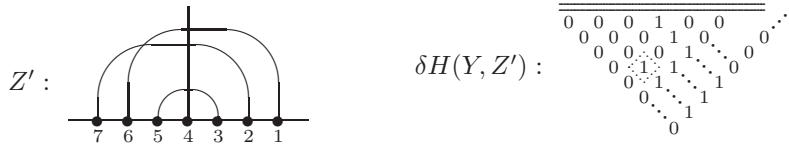
\begin{picture}(32,15)(0,3)
\put(-8,8){$Z':$}
\put(0,4){\line(1,0){32}}
\multiput(3,3)(4,0)7{$\bullet$}
\put(14,4){\oval(20,20)[t]}
\put(18,6){\oval(20,20)[t]}
\put(16,4){\oval(8,8)[t]}
\put(16,4){\line(0,1){15}}
\put(8,4){\line(0,1){2}}
\put(28,4){\line(0,1){2}}
\put(4,1){\makebox[0mm]{$\scriptstyle 7$}}
\put(8,1){\makebox[0mm]{$\scriptstyle 6$}}
\put(12,1){\makebox[0mm]{$\scriptstyle 5$}}
\put(16,1){\makebox[0mm]{$\scriptstyle 4$}}
\put(20,1){\makebox[0mm]{$\scriptstyle 3$}}
\put(24,1){\makebox[0mm]{$\scriptstyle 2$}}
\put(28,1){\makebox[0mm]{$\scriptstyle 1$}}
\end{picture}
\qquad\qquad
\beginpicture\setcoordinatesystem units <.4cm,.4cm>
\multiput{} at 0 0  4 4 /
\put{$\delta H(Y,Z'):$} at -3 2
\plot -.3 3.9  6.5 3.9 /
\plot -.3 4.1  6.5 4.1  /
\multiput{$\scriptstyle 0$} at 0 3.5  1 3.5  2 3.5  4 3.5  5 3.5  
                                  .5 3  1.5 3  4.5 3  6.5 3         2.5 3 
                                       1 2.5  3 2.5  6 2.5          2 2.5
                                           1.5 2  5.5 2
                                             2 1.5 
                                                 2.5 1
                                                    3.5 0 /
\multiput{$\scriptstyle 1$} at 3 3.5  
                                      3.5 3
                                      4 2.5
                               2.5 2  3.5 2
                               3 1.5  5 1.5
                                      4.5 1
                                      4 .5 /
\multiput{$\cdot$} at        2.8 .7  3 .5  3.2 .3
                             3.3 1.2  3.5 1  3.7 .8 
                             3.8 1.7  4 1.5  4.2 1.3
                             4.3 2.2  4.5 2  4.7 1.8 
                             4.8 2.7  5 2.5  5.2 2.3
                             6.8 3.3  7 3.5  7.2 3.7 /
\setdots<2pt>
\plot 2 2  2.5 2.5  3 2  2.5 1.5  2 2 /
\endpicture
$$

We pick a new parallelogram satisfying conditions (1) and (2)
(alternatively, we could have chosen the parallelogram 
given by the upper diagonal in the region marked by the 1's)
and perform the operation indicated:  Replace in $Z'$ 
the bipickets $B_2^{6,1}$ and $B_2^{7,2}$ by bipickets $B_2^{6,2}$ and $B_2^{7,1}$ 
in $Z''$.  Arc diagram and Hom-matrix are as follows.

$$
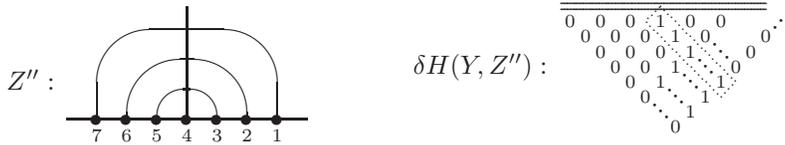
\begin{picture}(32,15)(0,3)
\put(-8,8){$Z'':$}
\put(0,4){\line(1,0){32}}
\multiput(3,3)(4,0)7{$\bullet$}
\put(16,4){\oval(24,24)[t]}
\put(16,4){\oval(16,16)[t]}
\put(16,4){\oval(8,8)[t]}
\put(16,4){\line(0,1){15}}
\put(4,1){\makebox[0mm]{$\scriptstyle 7$}}
\put(8,1){\makebox[0mm]{$\scriptstyle 6$}}
\put(12,1){\makebox[0mm]{$\scriptstyle 5$}}
\put(16,1){\makebox[0mm]{$\scriptstyle 4$}}
\put(20,1){\makebox[0mm]{$\scriptstyle 3$}}
\put(24,1){\makebox[0mm]{$\scriptstyle 2$}}
\put(28,1){\makebox[0mm]{$\scriptstyle 1$}}
\end{picture}
\qquad\qquad
\beginpicture\setcoordinatesystem units <.4cm,.4cm>
\multiput{} at 0 0  4 4 /
\put{$\delta H(Y,Z''):$} at -3 2
\plot -.3 3.9  6.5 3.9 /
\plot -.3 4.1  6.5 4.1  /
\multiput{$\scriptstyle 0$} at 0 3.5  1 3.5  2 3.5  4 3.5  5 3.5  
                                  .5 3  1.5 3  4.5 3  6.5 3         2.5 3 
                                       1 2.5  3 2.5  6 2.5          2 2.5
                                           1.5 2  5.5 2             2.5 2
                                             2 1.5 
                                                 2.5 1
                                                    3.5 0 /
\multiput{$\scriptstyle 1$} at 3 3.5  
                                      3.5 3
                                      4 2.5
                                      3.5 2
                               3 1.5  5 1.5
                                      4.5 1
                                      4 .5 /
\multiput{$\cdot$} at        2.8 .7  3 .5  3.2 .3
                             3.3 1.2  3.5 1  3.7 .8 
                             3.8 1.7  4 1.5  4.2 1.3
                             4.3 2.2  4.5 2  4.7 1.8 
                             4.8 2.7  5 2.5  5.2 2.3
                             6.8 3.3  7 3.5  7.2 3.7 /
\setdots<2pt>
\plot 3 4  5.5 1.5  5 1  2.5 3.5  3 4 /
\endpicture
$$

The parallelogram suggests an operation of type {\bf (B')}, namly 
to replace the bipicket $B_2^{5,3}$ and the picket $P_1^4$
in $Z''$ by pickets $P_2^5$, $P_0^4$ and $P_1^3$ in $Z'''$
(note that the first two pickets represent an arc from 5 to 4,
the last is a pole at 3).

$$
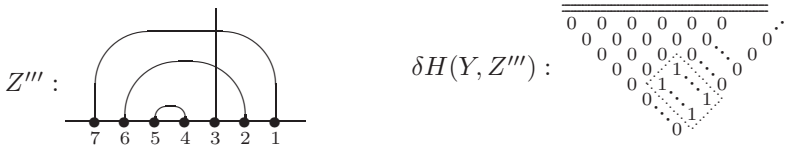
\begin{picture}(32,15)(0,3)
\put(-8,8){$Z''':$}
\put(0,4){\line(1,0){32}}
\multiput(3,3)(4,0)7{$\bullet$}
\put(16,4){\oval(24,24)[t]}
\put(16,4){\oval(16,16)[t]}
\put(14,4){\oval(4,4)[t]}
\put(20,4){\line(0,1){15}}
\put(4,1){\makebox[0mm]{$\scriptstyle 7$}}
\put(8,1){\makebox[0mm]{$\scriptstyle 6$}}
\put(12,1){\makebox[0mm]{$\scriptstyle 5$}}
\put(16,1){\makebox[0mm]{$\scriptstyle 4$}}
\put(20,1){\makebox[0mm]{$\scriptstyle 3$}}
\put(24,1){\makebox[0mm]{$\scriptstyle 2$}}
\put(28,1){\makebox[0mm]{$\scriptstyle 1$}}
\end{picture}
\qquad\qquad
\beginpicture\setcoordinatesystem units <.4cm,.4cm>
\multiput{} at 0 0  4 4 /
\put{$\delta H(Y,Z'''):$} at -3 2
\plot -.3 3.9  6.5 3.9 /
\plot -.3 4.1  6.5 4.1  /
\multiput{$\scriptstyle 0$} at 0 3.5  1 3.5  2 3.5  4 3.5  5 3.5      3 3.5
                                  .5 3  1.5 3  4.5 3  6.5 3    2.5 3  3.5 3
                                       1 2.5  3 2.5  6 2.5     2 2.5  4 2.5
                                           1.5 2  5.5 2        2.5 2 
                                             2 1.5                    5 1.5
                                                 2.5 1
                                                    3.5 0 /
\multiput{$\scriptstyle 1$} at %
                                      3.5 2
                               3 1.5  
                                      4.5 1
                                      4 .5 /
\multiput{$\cdot$} at        2.8 .7  3 .5  3.2 .3
                             3.3 1.2  3.5 1  3.7 .8 
                             3.8 1.7  4 1.5  4.2 1.3
                             4.3 2.2  4.5 2  4.7 1.8 
                             4.8 2.7  5 2.5  5.2 2.3
                             6.8 3.3  7 3.5  7.2 3.7 /
\setdots<2pt>
\plot 2.5 1.5  3.5 2.5   5 1  4 0  2.5 1.5 / 
\endpicture
$$

Finally, an operation of type {\bf (B)} reduces the Hom-matrix to zero 
and yields the module $Y$:  We replace the bipicket $B_2^{7,1}$ and the 
picket $P_1^3$ in $Z'''$ by the bipicket $B_2^{7,3}$ and the picket $P_1^1$ 
for $Y$.  Done!

$$Y\leq_{\rm arc} Z''' \leq_{\rm arc}  Z'' \leq_{\rm arc}  Z' \leq_{\rm arc} Z $$

\subsection{Operations on the arc diagram}
%-----------------------------------------

In this section we complete the proof of Theorem~\ref{theorem-hom-arc}.
Throughout we assume that $Y,Z\in\mathcal S_2(k)$ have the same 
partition type $(\alpha,\beta,\gamma)$.
To prove the implication: $Y\leq_{\rm hom} Z\;\Rightarrow\;Y\leq_{\rm arc}Z$, 
we assume that $Y\leq_{\rm hom}Z$ holds and apply the following result
repeatedly.

\begin{prop}\label{proposition-key-lemma}
Suppose $Y$ and $Z$ have the same 
partition type, $Y\leq_{\rm hom}Z$ and
$Y\not\cong Z$.  Then there is an operation on the arc diagram for $Z$
of type {\bf (A)}, {\bf (B)}, {\bf (C)} or {\bf (D)} 
which yields a module $Z'$ such that
$$ Z'\leq_{\rm arc} Z,\qquad Z'\not\cong Z,\qquad\text{and}\qquad
   Y\leq_{\rm hom}Z'.$$
\end{prop}

In the proof we will use the following result.

\begin{lem}\label{lemma-positive-region}
Consider the following matrix of integers.
$$\beginpicture\setcoordinatesystem units <.5cm,.5cm>
\multiput{} at 0 0  7 7 /
\put{$\beta^{0,0}$} at 0 3 
\put{$\beta^{1,0}$} at 1 2
\put{$\beta^{u,0}$} at 3 0
\put{$\beta^{0,v-1}$} at 2 5
\put{$\beta^{1,v-1}$} at 3 4
\put{$\beta^{u,v-1}$} at 5 2
\put{$\beta^{0,v}$} at 3 6
\put{$\beta^{1,v}$} at 4 5
\put{$\beta^{u,v}$} at 6 3
\put{$\beta^{0,v+1}$} at 4.2 7.2
\put{$\beta^{1,v+1}$} at 5.2 6.2
\put{$\beta^{u,v+1}$} at 7.2 4.2
\multiput{$\cdot$} at .8 3.8  1 4  1.2 4.2    
                      1.8 2.8  2 3  2.2 3.2
                      3.8 .8  4 1  4.2 1.2 /
\multiput{$\ddots$} at 2 1  4 3  5 4  6 5 /
\setdots<2pt>
\plot -.9 3.1  2.9 7.1  7 3  3 -1  -1 3 /
\endpicture$$
Suppose the following conditions are satisfied.
\begin{enumerate}
\item All entries are nonnegative
\item The numbers $\beta^{0,0}, \beta^{0,1},\ldots,\beta^{0,v}$ are 
      strictly positive
\item $\beta^{0,v+1}=0$
\item For each $0\leq i\leq u-1$, $0\leq j\leq v$,
$$\delta^{ij}=\beta^{i,j}+\beta^{i+1,j+1}-\beta^{i,j+1}-\beta^{i+1,j}\leq 0$$
\end{enumerate}
Then all entries in the parallelogram are positive:
$\beta^{i,j}>0$ for each $0\leq i\leq u$, $0\leq j\leq v$.
\end{lem}

\begin{proof}
By assumption, $\beta^{0,v},\beta^{0,v-1},\ldots,\beta^{0,0}$ are all strictly
positive.
Note that 
$$\beta^{1,v}=\beta^{0,v}+\beta^{1,v+1}-\beta^{0,v+1}-\delta^{0,v}\geq\beta^{0,v}>0.$$
Since $\delta^{0,j}+\cdots+\delta^{0,v}=\beta^{0,j}+\beta^{1,v+1}-\beta^{0,v+1}-\beta^{1,j}$, it follows that for each $0\leq j\leq v$,
$$\beta^{1,j}=\beta^{0,j}+\beta^{1,v+1}-\beta^{0,v+1}-
       (\delta^{0,j}+\cdots+\delta^{0,v})\geq\beta^{0,j}>0.$$
In general, since for each $1\leq i\leq u$, $0\leq j\leq v$,
$$\delta(i,j)=\sum_{0\leq h<i,j\leq\ell\leq v} \delta^{h,\ell}=
             \beta^{0,j}+\beta^{i,v+1}-\beta^{i,j}-\beta^{0,v+1}\leq 0,$$
we have
$$\beta^{i,j}=\beta^{0,j}+\beta^{i,v+1}-\beta^{0,v+1}-\delta(i,j)\geq \beta^{0,j}>0.$$
\end{proof}

Also the dual version holds:

\begin{lem}\label{lemma-positive-region-ii}
Consider the following matrix of integers.
$$
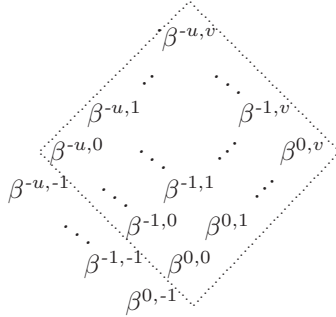
\beginpicture\setcoordinatesystem units <.5cm,.5cm>
\multiput{} at 0 0  7 7 /
\put{$\beta^{0,v}$} at 7 4
\put{$\beta^{\text-1,v}$} at 6 5
\put{$\beta^{\text-u,v}$} at 4 7
\put{$\beta^{0,1}$} at 5 2
\put{$\beta^{\text-1,1}$} at 4 3
\put{$\beta^{\text-u,1}$} at 2 5
\put{$\beta^{0,0}$} at 4 1
\put{$\beta^{\text-1,0}$} at 3 2 
\put{$\beta^{\text-u,0}$} at 1 4
\put{$\beta^{0,\text-1}$} at 3 0
\put{$\beta^{\text-1,\text-1}$} at 2 1
\put{$\beta^{\text-u,\text-1}$} at 0 3
\multiput{$\cdot$} at 6.2 3.2  6 3  5.8 2.8  5.2 4.2  5 4  4.8 3.8  3.2 7.2  3 6  2.8 5.8 /
\multiput{$\ddots$} at 5 6  3 4  2 3  1 2 /
\setdots<2pt>
\plot 7.9 3.9  4.1 -.1  0 4  4 8  8 4 /
\endpicture$$
Suppose that in addition to conditions 1 and 2 from the previous lemma also the following are satisfied.
\begin{enumerate}
\item[3'.] $\beta^{0,-1}=0$
\item[4'.] For each $-u\leq i<0$, $-1\leq j<v$,
$$\delta^{ij}=\beta^{i,j}+\beta^{i+1,j+1}-\beta^{i,j+1}-\beta^{i+1,j}\leq 0$$
\end{enumerate}
Then all entries in the parallelogram are positive:
$\beta^{i,j}>0$ for each $-u\leq i\leq 0$, $0\leq j\leq v$.
\qed
\end{lem}

\begin{proof}[Proof of Proposition~\ref{proposition-key-lemma}]

\smallskip
{\bf The set-up.} We assume that the entries in the Hom-matrix
$\delta H(Y,Z)$ are all nonnegative and that at least one entry is positive.

\smallskip
{\bf The goal.} We show that there is a parallelogram in the
Hom-matrix (in the shape of one of the shaded 
regions in Section~\ref{section-operations})
which satisfies the following two conditions.
\begin{itemize}
\item[(P1)] All entries within the parallelogram are strictly positive.
\item[(P2)] The two indecomposable modules $X'$ and $X''$ corresponding to the 
  right corner of the parallelogram and to the point just left of the
  left corner, respectively, occur with higher multiplicity as direct summands of $Z$
  than as a direct summand of $Y$.
\end{itemize}

For illustration, we repeat the situation before the first move in
the example.  The parallelogram is indicated.
$$
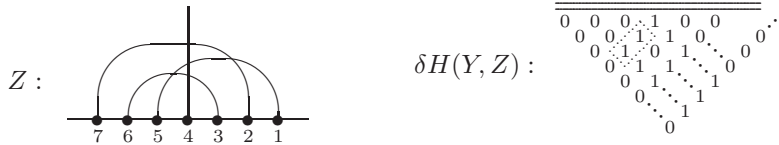
\begin{picture}(32,15)(0,3)
\put(-8,8){$Z:$}
\put(0,4){\line(1,0){32}}
\multiput(3,3)(4,0)7{$\bullet$}
\put(14,4){\oval(20,20)[t]}
\put(20,4){\oval(16,16)[t]}
\put(14,4){\oval(12,12)[t]}
\put(16,4){\line(0,1){15}}
\put(4,1){\makebox[0mm]{$\scriptstyle 7$}}
\put(8,1){\makebox[0mm]{$\scriptstyle 6$}}
\put(12,1){\makebox[0mm]{$\scriptstyle 5$}}
\put(16,1){\makebox[0mm]{$\scriptstyle 4$}}
\put(20,1){\makebox[0mm]{$\scriptstyle 3$}}
\put(24,1){\makebox[0mm]{$\scriptstyle 2$}}
\put(28,1){\makebox[0mm]{$\scriptstyle 1$}}
\end{picture}
\qquad\qquad
\beginpicture\setcoordinatesystem units <.4cm,.4cm>
\multiput{} at 0 0  4 4 /
\put{$\delta H(Y,Z):$} at -3 2
\plot -.3 3.9  6.5 3.9 /
\plot -.3 4.1  6.5 4.1  /
\multiput{$\scriptstyle 0$} at 0 3.5  1 3.5  2 3.5  4 3.5  5 3.5  
                               .5 3  1.5 3  4.5 3  6.5 3 
                               1 2.5  3 2.5  6 2.5
                               1.5 2  5.5 2
                               2 1.5 
                               2.5 1
                               3.5 0 /
\multiput{$\scriptstyle 1$} at 3 3.5  
                               2.5 3  3.5 3
                               2 2.5  4 2.5
                               2.5 2  3.5 2
                               3 1.5  5 1.5
                                      4.5 1
                                      4 .5 /
\multiput{$\cdot$} at        2.8 .7  3 .5  3.2 .3
                             3.3 1.2  3.5 1  3.7 .8 
                             3.8 1.7  4 1.5  4.2 1.3
                             4.3 2.2  4.5 2  4.7 1.8 
                             4.8 2.7  5 2.5  5.2 2.3
                             6.8 3.3  7 3.5  7.2 3.7 /
\setdots<2pt>
\plot 2.5 3.5  1.5 2.5  2 2  3 3  2.5 3.5 /
\endpicture
$$

\smallskip
{\bf Step 1.}
We choose a number $n>\beta_1$ and work in the category $\mathcal S_2^n$.
Recall that in the Hom matrix, the nonzero entries are confined to the 
orbits given by the modules $B_2^{3,1},\ldots,B_2^{n,1}$ which form a stripe of 
type $\mathbb Z\mathbb A_{n-2}$, with identifications. 
For the purpose of this algorithm, we view the Hom matrix as a stripe of
type $\mathbb Z\mathbb A_n$, with the orbits of the $B_2^{3,1},\ldots,B_2^{n,1}$
in the center and two orbits of zeros on the boundary.

\smallskip
Pick a sequence $\beta^{0,0},\ldots,\beta^{0,v}$ of positive entries 
arranged along an anti-diagonal in the Hom matrix, such that the neighboring entries
on the anti-diagonal, $\beta^{0,\text-1}$ and $\beta^{0,v+1}$, are both zero. 
This is possible since the entries in the matrix are nonnegative, since there
is at least one positive entry and since there are zeros on the boundary orbits.

\smallskip
In the example, we have picked the anti-diagonal 
$\beta^{0,0}=\delta H_{B_2^{6,2}}=1,\beta^{0,1}=\delta H_{B_2^{6,3}}=1,\beta^{0,2}=\delta H_{B_2^{6,4}}=1$.

\smallskip
{\bf Step 2.}
We determine the right corner of the parallelogram.  Put $u=0$.

\smallskip
Consider the sequence $\delta^{u,0},\ldots,\delta^{u,v}$ 
of entries in the multiplicity matrix in the positions given by 
the $\beta^{u,0},\ldots,\beta^{u,v}$ in the Hom matrix.

\smallskip
If one of the entries in the sequence $\delta^{u,0},\ldots,\delta^{u,v}$
is positive, put $u''=u$ and let $\delta^{u'',w''}$ be the first such entry. 
In this case, $(u'',w'')$ will be the right corner of the rectangle and
$X''$, the object corresponding to the position $(u'',w'')$, will be a summand
occurring with higher multiplicity in $Z$ than in $Y$.  We are done with the second step.

\smallskip
By Lemma~\ref{lemma-positive-region}, we obtain that all the numbers
$\beta^{u+1,0},\ldots,\beta^{u+1,v}$ in the Hom matrix, on the anti-diagonal just
under the previous anti-diagonal, are positive. 
Put $u:=u+1$ and proceed with the paragraph under Step~2.

\smallskip
Note that this process terminates:  For $u$ large enough, $\beta^{u,0}$
will correspond to a point on the boundary of the Hom matrix.  Hence $\beta^{u,0}=0$.

\smallskip
In the example, the number $\delta^{0,2}=\delta M_{B_2^{6,3}}=\beta^{0,2}+\beta^{1,3}-\beta^{0,3}-\beta^{1,2}=1$
is the first positive value among the entries in the multiplicity matrix on the 
anti-diagonal.  Thus, $X''=B_2^{6,3}$.

\smallskip
{\bf Step 3.}
We determine the left corner of the parallelogram.  Put $u=0$ and $v=v''$.

\smallskip
Consider the sequence $\delta^{u-1,-1},\ldots,\delta^{u-1,v-1}$ 
of entries in the multiplicity matrix in the positions just left of 
the $\beta^{u,0},\ldots,\beta^{u,v}$ in the Hom matrix.

\smallskip
If one of the entries in the sequence $\delta^{u-1,-1},\ldots,\delta^{u-1,v-1}$
is positive, put $u'=u-1$ and let $\delta^{u',w'}$ be the last such entry. 
In this case, $(u',w')$ will be the point just left of the left corner of the rectangle and
$X'$, the object corresponding to the position $(u',w')$, will be a summand
occurring with higher multiplicity in $Z$ than in $Y$.  We are done with the third step.

\smallskip
If none of the entries in $\delta^{u-1,-1},\ldots,\delta^{u-1,v-1}$ is positive, 
then by Lemma~\ref{lemma-positive-region-ii}, all the numbers
$\beta^{u-1,0},\ldots,\beta^{u-1,v}$ in the Hom matrix, on the anti-diagonal just
above the previous anti-diagonal, are positive. 
Put $u:=u-1$ and proceed with the paragraph under Step~3.

\smallskip
Note that this process terminates: For $-u$ large enough, the position corresponding to
$\beta^{u,0}$ will be a point on the boundary of the Hom matrix, so $\beta^{u,0}=0$.

\smallskip
In the example, $\delta^{-1,-1}=\delta M_{B_2^{5,1}}=1$ is the last (and first) positive 
entry in the multiplicity matrix on the line above the original anti-diagonal.
Thus, $X'=B_2^{5,1}$.

\smallskip
{\bf Conclusion.}
In each case the parallelogram marked off by $X'$ and $X''$ 
is of one of the types {\bf (A)}, {\bf (B)}, {\bf (C)} or {\bf (D)} in Section~\ref{section-operations}.
(Namely, if none of the modules $P_1^r$ is in the rectangle, then the type is {\bf (A)};
if $X''$ is one of the $P_1^r$, then the type is {\bf (B)}; in case $X'\cong P_1^r$ for some $r$,
then the type is {\bf (D)}; and finally, if there are modules $P_1^r$ in the rectangle, 
but neither $X'$ nor $X''$ has this form, then the type is {\bf (C)}.)
Write $Z=Z_0\oplus X'\oplus X''$  and let $X_1,\ldots,X_s$ ($s=2$ or $3$)
be the  modules corresponding to the entry $-1$ in the multiplicity matrix
for the arc operation.
Put $Z'=Z_0\oplus X_1\oplus\cdots\oplus X_s$.  By this construction,
$Z'\leq_{\rm arc} Z$.  Moreover,
$$\delta H(Y,Z')_X\;=\;\delta H(Y,Z)_X+\left\{\begin{array}{ll}-1 &
                            \text{if $X$ is in the parallelogram}\\
                     0 & \text{otherwise}\end{array}\right.$$
Since the entries for $\delta H(Y,Z)$ within the parallelogram are all
positive, the matrix $\delta H(Y,Z')$ is nonnegative and hence
$$Y\leq_{\rm hom} Z'\qquad\text{and}\qquad Z'\leq_{\rm arc}Z.$$
\end{proof}

\begin{proof}[Proof of Theorem~\ref{theorem-hom-arc}]
Assume that $Y\leq_{\rm hom}Z$.
Since each application of Proposition~\ref{proposition-key-lemma}
reduces the number of crossings in the arc diagram,
a finite number of steps
suffices to produce a sequence of modules 
$Z$, $Z'$, $Z''$, $\ldots$, $Z^{(m)}=Y$ such that
$$Y=Z^{(m)}\leq_{\rm arc} Z^{(m-1)}\leq_{\rm arc} \cdots\leq_{\rm arc} Z''
          \leq_{\rm arc} Z' \leq_{\rm arc} Z.$$
\end{proof}

\section{Dimensions via Hall polynomials}\label{sec-dim-hall}
%=======================================

In this section we investigate the set $V_{\alpha,\gamma}^\beta$.
We recall that if $k$ is an algebraically closed field,
$$V_{\alpha,\gamma}^\beta(k) \;=\;  \bigcup_\Gamma \; V_\Gamma(k) \;=\; 
\bigcup_\Gamma\bigcup_{\Pi} \; V_\Pi(k),$$
where the first union is indexed by all LR-tableaux $\Gamma$
of type $(\alpha,\beta,\gamma)$ and the second
union is indexed by all Klein tableaux $\Pi$ refining $\Gamma$,
and where $V_\Pi(k)=G_f$ is the orbit of a~point 
$(N_\alpha(k),N_\beta(k),f)\in V_{\alpha,\gamma}^\beta(k)$ with Klein tableau $\Pi$.

\smallskip
Correspondingly, if $k$ is a finite field of $q$ elements, there is the following
sum formula for Hall polynomials.
$$h_{\alpha,\gamma}^\beta(q) \;=\;  \sum_\Gamma h_\Gamma(q) \;=\; 
\sum_\Gamma\sum_{\Pi}h_\Pi(q),$$
where the indices are as above.  The polynomials $h_\Gamma$ are monic of the same degree
$n(\beta)-n(\alpha)-n(\gamma)$, while the polynomials $h_\Pi$ are monic of degree
$n(\beta)-n(\alpha)-n(\gamma)-x(\Pi)$ where $x(\Pi)$ denotes the deviation from dominance
\cite[Corollaries~1-3]{klein}.

\subsection{$\mathbb{Z}$-forms of objects}\label{subsec-Z-form}
%-----------------------------------------

The following lemma 
shows that the categories $\mathcal{N}(k)$ and $\mathcal{S}_2(k)$
can be defined over $\mathbb{Z}$ (independently on the field $k$, see \cite{kas}, \cite{kaskos}).
Denote by $\mathcal{N}(\mathbb{Z})$ the category of all systems $(V,\varphi)$,
where $V$ is a~finitely generated free abelian group and $\varphi:V\to V$
is a~nilpotent $\mathbb{Z}$-homomorphism. The morphisms are
defined in the usual way. For any field $k$, let 
$\k{(V,\varphi)}=(V\otimes_{\mathbb{Z}} k, \varphi \otimes_{\mathbb{Z}} \id)$.
Similarly, denote by $\mathcal{S}(\mathbb{Z})$ the category of all systems 
$(V,W,\psi)$, where $V$, $W$ are objects in $\mathcal{N}(\mathbb{Z})$ and $\psi:V\to W$
is a~morphism in $\mathcal{N}(\mathbb{Z})$. 
Morphisms in $\mathcal S(\mathbb Z)$ and the functor
$$-\otimes k:\qquad \mathcal S(\mathbb Z)\;\to\; \mathcal S(k)$$
are defined analogously.

\begin{lem}\label{lemma-defined-over-Z}
{\rm (1)} For any partition $\alpha$ there exists an~object 
$N_\alpha(\mathbb{Z})=(V,\varphi)$ in $\mathcal{N}(\mathbb{Z})$, such that for any field $k$: 
$$\k{N_\alpha(\mathbb{Z})}   \;\cong\;   N_\alpha(k).$$ 
Moreover there exists a~$\mathbb{Z}$-basis of $V$ such that the matrix of $\varphi$
in this basis has coefficients equal to $0$ or $1$.

{\rm (2)} For any Klein tableau $\Pi$ with entries at most two 
there exists an~object $M_\Pi(\mathbb{Z})=(V,W,\psi)$,
in $\mathcal{S}(\mathbb{Z})$, such that for any field $k$: 
$$\k{M_\Pi(\mathbb{Z})} \;\cong\;   M_\Pi(k).$$ 
Moreover there exist $\mathbb{Z}$-bases of $V$ and $W$ such that the matrix of $\psi$
in these bases has coefficients equal to $0$ or $1$.
\label{lem:Z-form} \end{lem} 

\begin{proof}
(1) Let $\alpha=(\alpha_1\geq\ldots\geq\alpha_n)$ be a~partition.
Set $N_\alpha(\mathbb{Z})=(V,\varphi)$, where $V=\mathbb{Z}^{\alpha_1}\oplus\ldots\oplus\mathbb{Z}^{\alpha_n}$,
and $\varphi$ is a~direct sum of nilpotent Jordan blocks of sizes $\alpha_1,\ldots,\alpha_n$.
It is clear that $N_\alpha(\mathbb{Z})$ satisfies the required conditions.

(2)  Let $\Pi$ be a~Klein tableau with entries at most two and let $k$ be an~arbitrary field. 
By Theorem \ref{thm-pic-bipic} the module
$M_\Pi(k)$ is a~direct sum of pickets $P_\ell^m(k)$ 
and bipickets $B_2^{m,r}(k)$. We set 
$$P_\ell^m(\mathbb{Z})=(N_{(\ell)}(\mathbb{Z}),N_{(m)}(\mathbb{Z}),\iota)$$ 
and 
$$B_2^{m,r}(\mathbb{Z})=(N_{(2)}(\mathbb{Z}),N_{(m,r)}(\mathbb{Z}),\delta).$$
Note that there exist bases such that the matrices of $\iota$ and $\delta$ 
have entries equal to $0$ and $1$.
Moreover $\k{P_\ell^m(\mathbb{Z})}\cong P_\ell^m(k)$ and $\k{B_2^{m,r}(\mathbb{Z})}\cong B_2^{m,r}(k)$.
This finishes the proof.    
\end{proof}

\subsection{Dimensions of orbits}\label{subsec-dim-orbit}
%--------------------------------------------------------

\smallskip
Let now $k$ be an algebraically closed field,
$\alpha$, $\beta$, $\gamma$ partitions with $\alpha_1\leq 2$,
let $\Gamma$ be an~LR-tableau of type $(\alpha,\beta,\gamma)$,
$\Pi$ a Klein tableau refining $\Gamma$, 
and $M_\Pi(k)=(N_\alpha,N_\beta,f)$ an object in $\mathcal{S}_2(k)$ 
corresponding to $\Pi$.
In Section \ref{sec-four-partial} we described an~algebraic group action
under which $V_\Pi(k)=G_f$ is an orbit.  Then we have
  \begin{equation} \dim\, V_\Pi(k) = \dim \Aut_{\mathcal{N}}(N_\alpha(k))
    + \dim\Aut_{\mathcal{N}}(N_\beta(k)) -\dim \Aut_{\mathcal{S}}(M_\Pi(k)),
   \label{eq:dimension}  \end{equation}
where $\dim$ is the variety dimension. The subset $V_\Pi(k)$
of  $V_\alpha^\beta(k)$ is locally closed, because it is an~orbit
of an~action of an~algebraic group. Therefore $V_\Pi(k)$ is locally closed in
$H_\alpha^\beta(k)$.

Recall the result of Lang and Weil \cite{lw} (see also \cite[Proposition 5.6]{kr}).
Let $\mathbb{F}_p$ be a~field with $p$ elements, let $k=\ov{\mathbb{F}_p}$ be its algebraic closure
and let $n\in \mathbb{N}$. 
Moreover, let $\mathcal{U}\subseteq k^n$ be a~locally closed subset which is closed
under the Frobenius automorphism $\sigma$ of $k^n$ (i.e.\ under the  morphism
$\sigma:k^n\to k^n$ defined by $\sigma(x)=\sigma(x_1,\ldots,x_n)=(x_1^p,\ldots,x_n^p)=x^{(p)}$). 
For any finite subfield
$\mathbb{F}_q$ of $k$,  set $$ \mathcal{U}(\mathbb{F}_q)=\mathcal U\cap \mathbb{F}_q^n. $$

\begin{prop}[\cite{lw},\cite{kr}]
 If $\mathcal{U}$ is closed under $\sigma$, then
  $$ \# \mathcal{U}(\mathbb{F}_q)\approx c\,q^{\dim\mathcal{U}},$$
where $c$ is the number of irreducible components of $\mathcal{U}$ of maximal dimension, and
$f(q)\approx g(q)$ if $\lim_{q\to\infty}\frac{f(q)-g(q)}{f(q)}=0$.
\label{prop:lw}\qed\end{prop} \vsp

We use this result to determine the dimension of the orbit $V_\Pi(k)$
as the degree of the Hall polynomial $g_\Pi(t)$.
Let 
$$\begin{array}{rcl}
  a_{\alpha}(q) &  = &  \#\Aut_{\mathcal{N}(\mathbb{F}_q)}(N_\alpha(\mathbb{F}_q)), \\[1ex]
  a_{\Pi}(q)   &  =  &  \#\Aut_{\mathcal{S}(\mathbb{F}_q)}(M_\Pi(\mathbb F_q)),\;\text{and} \\[1ex]
  g_{\Pi}(q)  &  =  &  \displaystyle \frac{a_\beta(q) \cdot  a_\alpha(q)}{a_\Pi(q)}.\end{array}
$$

Note that 
$$ \#V_\Pi(\mathbb{F}_q)\; =\;  g_{\Pi}(q),$$
and
$$ g_{\Pi}(q)=h_{\Pi}(q)\cdot a_{\alpha}(q) $$
for any $q$ (compare \cite[3.1]{sch}).

\begin{lem} The function $a_{\alpha}(q)$ is a~polynomial in $q$
with integral coefficients. If $\Pi$ is a Klein tableau of type $(\alpha,\beta,\gamma)$,
then $a_{\Pi}(q)$ and $g_{\Pi}(q)$ are polynomials in $q$
with integral coefficients.
\label{lem:poly}\end{lem}

\begin{proof}
By \cite[II.1.6]{macd} and \cite[Lemma 4]{sch}, $a_{\alpha}(q)$, $a_{\Pi}(q)$ are polynomials 
in $q$ with integral coefficients.
To prove that  $g_{\Pi}(q)$ is a~polynomial in $q$ with integral coefficients
one can develop arguments given in \cite[Section 2, Lemma (4)]{rin1}
\end{proof}

\begin{prop}
  Let $\alpha, \beta$ be  partitions, let $p$ be a~prime number
and let $k=\ov{\mathbb{F}_p}$. Then 
   \begin{enumerate}
       \item $ \dim \Aut_{\mathcal{N}}(N_\alpha(k)) = \deg a_{\alpha},$
   \end{enumerate}
If in addition $M_\Pi(k)=(N_\alpha(k), N_\beta(k),f)$ is an object in $\mathcal{S}_2(k)$ 
with Klein tableau $\Pi$, then 
   \begin{enumerate}
        \item[2.] $ \dim \Aut_{\mathcal{S}}(M_\Pi(k)) 
          = \deg a_{\Pi},$
       \item[3.] $ \dim V_\Pi(k) = \deg g_{\Pi}.$
   \end{enumerate}   
 \label{prop:dim-fin-char}
\end{prop}

\begin{proof}
We observed that the set $V_\Pi(k)$ is a~locally closed subset of the affine variety 
$H_\alpha^\beta(k)$. 
Note that the sets $\Aut_{\mathcal{N}}(N_\alpha(k))$, 
$\Aut_{\mathcal{S}}(M_\Pi(k))$ are locally closed subsets
of the affine varieties $\End_k(N_\alpha(k))$,
$\End_k(N_\alpha(k), N_\beta(k),f)$, respectively.
Applying Lemma \ref{lem:Z-form}, we prove  that 
$V_\Pi(k)$, $\Aut_{\mathcal{N}}(N_\alpha(k))$, 
$\Aut_{\mathcal{S}}(M_\Pi(k))$ are closed under the Frobenius automorphism $\sigma$.
Consider the subset  $\Aut_{\mathcal{N}}(N_\alpha(k))$ of $\End_k(N_\alpha(k))$. Note that a~matrix $X$
is in $\Aut_{\mathcal{N}}(N_\alpha(k))$ if and only if $\det X\neq 0$ and $X\cdot A=A\cdot X$,
where the matrix $A$ has entries $0$ or $1$ (apply Lemma \ref{lem:Z-form}).
Therefore $\det X^{(p)}\neq 0$ and $X^{(p)}A=(X\cdot A)^{(p)}=(A\cdot X)^{(p)}=A\cdot X^{(p)}$.
This proves that $X^{(p)}$ is in $\Aut_{\mathcal{N}}(N_\alpha(k))$. Similarly, we prove that
$\Aut_{\mathcal{S}}(M_\Pi(k))$ is closed under $\sigma$.
By Lemma \ref{lem:Z-form},  $X$ is in $V_\Pi(k)$ if and only if there exists
a~matrix $Y\in V_\Pi(k)$ with entries $1$ and $0$ and an~invertible matrix 
$A$ such that $X=A\cdot Y\cdot A^{-1}$.
Then  $X^{(p)}=(A\cdot Y\cdot A^{-1})^{(p)}=A^{(p)}\cdot Y\cdot (A^{(p)})^{-1}$. This shows that
$X^{(p)}$ is in $V_\Pi(k)$.

Therefore by Proposition \ref{prop:lw},
$$g_{\Pi}(q) \approx c q^{\dim G_k(\Pi)}.$$ Finally,
 $$ \dim V_\Pi(k) = \deg g_{\Pi}.$$ The remaining statements follow in a~similar way.
\end{proof}

\begin{prop}
  Let $\alpha, \beta$ be partitions and let $M_\Pi(k)=(N_\alpha(k), N_\beta(k),f)$ 
be an object in $\mathcal{S}_2(k)$ with Klein tableau $\Pi$. Then
   $$ \dim V_\Pi(k) = \deg g_{\Pi}$$
for any algebraically closed field $k$.
 \label{prop:dim-zero-char}
\end{prop}
 
\begin{proof}
Let $k$ be an~arbitrary algebraically closed field. 
It is well known that
\begin{itemize}
        \item $ \dim \Aut_{\mathcal{N}}(N_\alpha(k)) = \dim_k \End_{\mathcal{N}}(N_\alpha(k)),$
        \item $ \dim \Aut_{\mathcal{S}}(M_\Pi(k)) = \dim_k \End_{\mathcal{S}}(M_\Pi(k)).$
\end{itemize}
Developing Lemma \ref{lem:Z-form}, it is easy to see that the properties 
\begin{itemize}
        \item $ \dim_k \End_{\mathcal{N}}(N_\alpha(k)) = n_k,$
        \item $ \dim_k \End_{\mathcal{S}}(M_\Pi(k)) = m_k,$
\end{itemize}
can be written as first order formulae in the language of rings
\cite[Chapter~10]{JL2}. By Proposition~\ref{prop:dim-fin-char}, for any $k=\overline{\mathbb{F}_p}$ we have 
$n_k=\deg a_{\alpha}$ and $m_k=\deg a_{\Pi}$. Therefore these numbers do
not depend on the prime number $p$. The Characteristic Transfer
Principle \cite[Theorem 1.14]{JL2} and Proposition \ref{prop:dim-fin-char} applied to these formulae
yield that for any algebraically closed field $k$ we have 
\begin{itemize}
       \item $ \dim \Aut_{\mathcal{N}}(N_\alpha(k)) = \deg a_{\alpha},$
        \item $ \dim \Aut_{\mathcal{S}}(M_\Pi(k)) = \deg a_{\Pi},$
   \end{itemize}
Then Formula \ref{eq:dimension} proves that $$\dim V_\Pi(k) = \deg g_{\Pi}.$$
\end{proof}

\medskip
We can now complete the proof of Theorem~\ref{thm-second-main}.

% \begin{proof}
% For the proof that $\dim V_\Pi(k)=n(\beta)-n(\alpha)-n(\gamma)-x(\Delta)\blue{+|\alpha|+2n(\alpha)}$, 
% we have already seen that $\dim V_\Pi(k)=\deg g_\Pi$ (Proposition~\ref{prop:dim-zero-char}),
% that $\deg \red{g}\blue{h}_\Gamma=n(\beta)-n(\alpha)-n(\gamma)$ \cite[Corollary~2, p.~77]{klein},
% and that $x(\Delta)=x(\Pi)$ (Lemma~\ref{lemma-crossings}).
% \textcolor{blue}{Moreover $g_\Gamma=h_\Gamma\cdot a_{\alpha}$, $g_\Pi=h_\Pi\cdot a_{\alpha}$ and
% $\deg a_{\alpha}=|\alpha|+2n(\alpha)$.} 
% The remaining equality $\deg \red{g}\blue{h}_\Pi=\deg g_\Gamma - x(\Pi)$ is shown in \cite{klein}:
% In Corollary~1, the Hall polynomial is computed as
% $$\red{g}\blue{h}_\Pi(q)=q^{\deg g_\Gamma} f\big(\Pi,{\textstyle \frac 1q}\big).$$ 
% According to the paragraph leading to \cite[Definition~3.5]{klein}, 
% $q^{x(\Pi)}f(\Pi,\frac 1q)$ is a polynomial in $\frac 1q$ with 
% constant coefficient $1$.  Then $q^{\deg \red{g}\blue{h}_\Gamma} f\big(\Pi,\frac1q\big)$ is a monic
% polynomial in $q$ of degree $\deg \red{g}\blue{h}_\Gamma-x(\Pi)$.
% \end{proof}

\begin{proof}
%\begin{green}
For the proof that $\dim V_\Pi(k)=\deg h_{\alpha,\gamma}^\beta + \deg a_\alpha -x(\Delta)$,
we have already seen that $\dim V_\Pi(k)=\deg g_\Pi$ (Proposition~\ref{prop:dim-zero-char}).
Since each image of an embedding $N_\alpha(q)\to N_\beta(q)$ can be realized by 
$a_\alpha(q)$ many monomorphisms, we deduce
$$g_\Gamma=h_\Gamma\cdot a_{\alpha}, \quad\text{and}\quad g_\Pi=h_\Pi\cdot a_{\alpha}.$$
We recall from \cite[II, (1.6)]{macd} that $\deg a_{\alpha}=|\alpha|+2n(\alpha)$.
The equality $\deg h_\Pi=\deg h_\Gamma - x(\Pi)$ is shown in \cite{klein}:
In Corollary~1, the Hall polynomial is computed as
$$h_\Pi(q)=q^{\deg h_\Gamma} f\big(\Pi,{\textstyle \frac 1q}\big).$$ 
According to the paragraph leading to \cite[Definition~3.5]{klein}, 
$q^{x(\Pi)}f(\Pi,\frac 1q)$ is a polynomial in $\frac 1q$ with 
constant coefficient $1$.  Then $q^{\deg h_\Gamma} f\big(\Pi,\frac1q\big)$ is a monic
polynomial in $q$ of degree $\deg h_\Gamma-x(\Pi)$.
It remains to observe that 
$\deg h_\Gamma=\deg h_{\alpha,\gamma}^\beta=n(\beta)-n(\alpha)-n(\gamma)$ 
\cite[Corollary~2, p.~77]{klein},
and that $x(\Delta)=x(\Pi)$ (Lemma~\ref{lemma-crossings}).
%\end{green}
\end{proof}

\subsection{Properties of the partial order $\leq_{\rm arc}$} \label{subsec-lattice}
%----------------------------------------------------------

Let $\CK(\Gamma)$ 
be the set of all Klein tableaux refining an~LR-tableau $\Gamma$ with entries
at most two,
  and $\CK(\alpha,\beta,\gamma)$ the set of all Klein tableaux of partition type
  $(\alpha,\beta,\gamma)$ with $\alpha_1\leq 2$.
We describe properties of the posets $\CK(\Gamma)=(\CK(\Gamma),\leq_{\rm arc})$ 
and $\CK(\alpha,\beta,\gamma)=(\CK(\alpha,\beta,\gamma),\leq_{\rm arc})$.

\begin{thm}\label{thm-lattice} Let $\Gamma$ be an~LR tableau with entries
  at most two.
  \begin{enumerate}
  \item In the poset $\CK(\Gamma)$ there exists exactly one minimal element: the dominant Klein tableau.
  \item In the poset $\CK(\Gamma)$ there exists exactly one maximal element.
  \item In the poset $\CK(\alpha,\beta,\gamma)$ there exists exactly one maximal element.
  \item The set of minimal elements of the poset $\CK(\alpha,\beta,\gamma)$ 
    consists of the minimal elements of the posets
    $\CK(\Gamma_1),\ldots,\CK(\Gamma_s)$, where $\Gamma_1,\ldots,\Gamma_s$ 
    are all the LR-tableaux of type
    $(\alpha,\beta,\gamma)$.
  \end{enumerate} 
\end{thm}

\begin{proof}
1. Assume that $\Pi\in\CK(\Gamma)$ is not dominant. 
Recalling from Section~\ref{section-lr-tableau} how the dominant Klein tableau is obtained,
it follows that there exists a~row $m$ in $\Pi$ with a symbol $\singlebox{2_r}$, 
where $r$ is not the largest available subscript. Choose $m$ minimal with this property and 
denote by $r'$ the largest available subscript that could be assigned to $2$ in this row. 
Therefore in the arc diagram corresponding to $\Pi$ we have 

$$
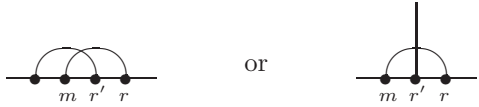
\begin{picture}(20,8)(0,5)
\put(0,4){\line(1,0){20}}
\multiput(3,3)(4,0)4{$\bullet$}
%\put(4,1){$\scriptstyle 4$}
\put(7,1){$\scriptstyle m$}
\put(11,1){$\scriptstyle r'$}
\put(15,1){$\scriptstyle r$}
\put(8,4){\oval(8,8)[t]}
\put(12,4){\oval(8,8)[t]}
\end{picture}
\quad\qquad \mbox{ or } \quad\qquad
\begin{picture}(16,8)(0,5)
\put(0,4){\line(1,0){16}}
\multiput(3,3)(4,0)3{$\bullet$}
\put(3,1){$\scriptstyle m$}
\put(7,1){$\scriptstyle r'$}
\put(11,1){$\scriptstyle r$}
\put(8,4){\oval(8,8)[t]}
\put(8,4){\line(0,1){10}}
\end{picture}
$$\vsp

\noindent This proves that $\Pi$ is not arc-minimal.  Then only the dominant 
tableau can be arc-minimal. 

2. We prove that the unique arc-maximal $\Pi\in\CK(\Gamma)$ is defined as follows: 
From an LR-tableau $\Gamma$ one can obtain a Klein tableau $\Pi$ by
working through the entries in $\Gamma$ row by row 
(starting at the top) and assigning
to each symbol $\ell\geq2$ the smallest available subscript 
(due to the lattice permutation property, 
there is always a subscript available). 

Assume that $\Pi\in\CK(\Gamma)$ does not satisfy this condition. 
Then there exists a~row $m$ with symbol $\singlebox{2_r}$, 
where $r$ is not the smallest available subscript. Choose $m$ minimal with this property and 
denote by $r'$ the smallest available subscript that could be assigned to $2$ in this row. 
Therefore in the arc diagram corresponding to $\Pi$ we have 

$$
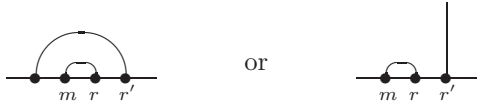
\begin{picture}(20,8)(0,5)
\put(0,4){\line(1,0){20}}
\multiput(3,3)(4,0)4{$\bullet$}
%\put(4,1){$\scriptstyle 4$}
\put(7,1){$\scriptstyle m$}
\put(11,1){$\scriptstyle r$}
\put(15,1){$\scriptstyle r'$}
\put(10,4){\oval(4,4)[t]}
\put(10,4){\oval(12,12)[t]}
\end{picture}
\quad\qquad \mbox{ or } \quad\qquad
\begin{picture}(16,8)(0,5)
\put(0,4){\line(1,0){16}}
\multiput(3,3)(4,0)3{$\bullet$}
\put(3,1){$\scriptstyle m$}
\put(7,1){$\scriptstyle r$}
\put(11,1){$\scriptstyle r'$}
\put(6,4){\oval(4,4)[t]}
\put(12,4){\line(0,1){10}}
\end{picture}
$$
\vsp

\noindent This proves that $\Pi$ is not arc-maximal.  

3. Define $\Pi\in\CK(\alpha,\beta,\gamma)$ as follows. 
Choose an~LR-tableau $\Gamma\in\CK(\alpha,\beta,\gamma)$ such that
the entries $\singlebox{2}$ are in the largest available rows. Let $\Pi$ be the arc-maximal
element in $\CK(\Gamma)$. It is easy to see that we cannot apply to $\Pi$ moves of types {\bf (C)}, {\bf (D)},
because these moves sent the entries $\singlebox{2}$ to larger rows.  Since  $\Pi$ is the arc-maximal
element in $\CK(\Gamma)$, we cannot apply to $\Pi$ moves of types {\bf (A)}, {\bf (B)}. Therefore
$\Pi$ is an~arc-maximal element in $\CK(\alpha,\beta,\gamma)$.

Assume that $\Pi'$ is an~element in $\CK(\alpha,\beta,\gamma)$ such that in the LR-tableau $\Gamma'$
corresponding to $\Pi'$ there exists an~entry $\singlebox{2}$ in the row $r$ that is not the largest possible. 
Therefore in the arc diagram corresponding to $\Pi'$ we have 

$$
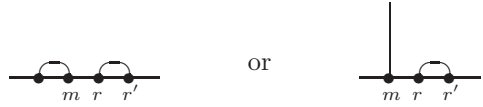
\begin{picture}(20,8)(0,5)
\put(0,4){\line(1,0){20}}
\multiput(3,3)(4,0)4{$\bullet$}
%\put(4,1){$\scriptstyle 4$}
\put(7,1){$\scriptstyle m$}
\put(11,1){$\scriptstyle r$}
\put(15,1){$\scriptstyle r'$}
\put(14,4){\oval(4,4)[t]}
\put(6,4){\oval(4,4)[t]}
\end{picture}
\quad\qquad \mbox{ or } \quad\qquad
\begin{picture}(16,8)(0,5)
\put(0,4){\line(1,0){16}}
\multiput(3,3)(4,0)3{$\bullet$}
\put(3,1){$\scriptstyle m$}
\put(7,1){$\scriptstyle r$}
\put(11,1){$\scriptstyle r'$}
\put(10,4){\oval(4,4)[t]}
\put(4,4){\line(0,1){10}}
\end{picture}
$$
\vsp

\noindent This proves that $\Pi'$ is not arc-maximal.  

4. Let $\Pi_1,\ldots,\Pi_s$ be the minimal elements in the posets
$\CK(\Gamma_1),\ldots,\CK(\Gamma_s)$, respectively. Since the orbits 
$G_{\Pi_1}(k),\ldots,G_{\Pi_s}(k)$ have maximal dimension equal to
$\deg\, g_{\alpha,\gamma}^\beta$ and 
since the arc-order is equivalent to the deg-order, the elements $\Pi_1,\ldots,\Pi_s$
are minimal in $\CK(\alpha,\beta,\gamma)$. It is easy to see that they are the 
only minimal elements. 
\end{proof}%\epv

\smallskip
{\it Example 1:}
Consider the following LR-tableau  of type $(\alpha,\beta,\gamma)$,
where $\alpha=(2,2,1,1)$, $\gamma=(4,3,2,2,1)$ and $\beta=(5,4,3,3,2,1)$

\vspace{-4mm}
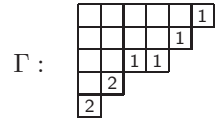
$$\Gamma:\quad
\setcounter{boxsize}{3}
\begin{picture}(18,12)(0,6)
\multiput(0,12)(3,0)5{\smbox}
\put(15,12){\numbox{1}}
\multiput(0,9)(3,0)4{\smbox}
\put(12,9){\numbox{1}}
\multiput(0,6)(3,0)2{\smbox}
\multiput(6,6)(3,0)2{\numbox{1}}
\put(0,3){\smbox}
\put(3,3){\numbox{2}}
\put(0,0){\numbox{2}}
\end{picture}
$$
\vspace{0.5cm}

The arc diagrams $\Pi_1,\ldots,\Pi_7$ 
corresponding to $\Gamma$ are pictured in the Hasse diagram.

\begin{figure}[p]
$$
\begin{picture}(100,180)

\put(30,168){$\Pi_7:$}

\put(40,160){\framebox(26,18)[t]{\begin{picture}(24,12)(0,5)
\put(0,4){\line(1,0){24}}
\multiput(3,3)(4,0)5{$\bullet$}
\put(4,1){$\scriptstyle 5$}
\put(8,1){$\scriptstyle 4$}
\put(12,1){$\scriptstyle 3$}
\put(16,1){$\scriptstyle 2$}
\put(20,1){$\scriptstyle 1$}
\put(10,4){\oval(12,12)[t]}
\put(14,4){\oval(12,12)[t]}
\put(12,4){\line(1,1){10}}
\put(12,4){\line(-1,1){10}}
\end{picture}}}

\put(26,140){\vector(1,1){15}}
\put(78,140){\vector(-1,1){15}}
\put(-10,128){$\Pi_5:$}

\put(0,120){\framebox(26,18)[t]{\begin{picture}(24,12)(0,5)
\put(0,4){\line(1,0){24}}
\multiput(3,3)(4,0)5{$\bullet$}
\put(4,1){$\scriptstyle 5$}
\put(8,1){$\scriptstyle 4$}
\put(12,1){$\scriptstyle 3$}
\put(16,1){$\scriptstyle 2$}
\put(20,1){$\scriptstyle 1$}
\put(8,4){\oval(8,8)[t]}
\put(14,4){\oval(12,12)[t]}
\put(12,4){\line(0,1){10}}
\put(16,4){\line(0,1){10}}
\end{picture}}}

\put(70,128){$\Pi_6:$}
 
\put(80,120){\framebox(26,18)[t]{\begin{picture}(24,12)(0,5)
\put(0,4){\line(1,0){24}}
\multiput(3,3)(4,0)5{$\bullet$}
\put(4,1){$\scriptstyle 5$}
\put(8,1){$\scriptstyle 4$}
\put(12,1){$\scriptstyle 3$}
\put(16,1){$\scriptstyle 2$} 
\put(20,1){$\scriptstyle 1$}
\put(12,4){\oval(16,16)[t]}
\put(12,4){\oval(8,8)[t]}
\put(12,4){\line(1,1){10}}
\put(12,4){\line(-1,1){10}}
\end{picture}}}

\put(26,100){\vector(3,1){50}}
\put(78,100){\vector(-3,1){50}}
\put(13,100){\vector(0,1){15}}
\put(93,100){\vector(0,1){15}}
\put(-10,88){$\Pi_3:$}

\put(0,80){\framebox(26,18)[t]{\begin{picture}(24,12)(0,5)
\put(0,4){\line(1,0){24}}
\multiput(3,3)(4,0)5{$\bullet$}
\put(4,1){$\scriptstyle 5$}
\put(8,1){$\scriptstyle 4$}
\put(12,1){$\scriptstyle 3$}
\put(16,1){$\scriptstyle 2$}
\put(20,1){$\scriptstyle 1$}
\put(8,4){\oval(8,8)[t]}
\put(12,4){\oval(8,8)[t]}
\put(12,4){\line(0,1){10}}
\put(20,4){\line(0,1){10}}
\end{picture}}}

\put(70,88){$\Pi_4:$}

\put(80,80){\framebox(26,18)[t]{\begin{picture}(24,12)(0,5)
\put(0,4){\line(1,0){24}}
\multiput(3,3)(4,0)5{$\bullet$}
\put(4,1){$\scriptstyle 5$}
\put(8,1){$\scriptstyle 4$}
\put(12,1){$\scriptstyle 3$}
\put(16,1){$\scriptstyle 2$}
\put(20,1){$\scriptstyle 1$}
\put(12,4){\oval(16,16)[t]}
\put(10,4){\oval(4,4)[t]}
\put(12,4){\line(0,1){10}}
\put(16,4){\line(0,1){10}}
\end{picture}}}

\put(63,60){\vector(1,1){15}}
\put(43,60){\vector(-1,1){15}}
\put(30,48){$\Pi_2:$}

\put(40,40){\framebox(26,18)[t]{\begin{picture}(24,12)(0,5)
\put(0,4){\line(1,0){24}}
\multiput(3,3)(4,0)5{$\bullet$}
\put(4,1){$\scriptstyle 5$}
\put(8,1){$\scriptstyle 4$}
\put(12,1){$\scriptstyle 3$}
\put(16,1){$\scriptstyle 2$}
\put(20,1){$\scriptstyle 1$}
\put(10,4){\oval(12,12)[t]}
\put(10,4){\oval(4,4)[t]}
\put(12,4){\line(0,1){10}}
\put(20,4){\line(0,1){10}}
\end{picture}}
}
\put(53,20){\vector(0,1){15}}

\put(30,8){$\Pi_1:$}
\put(40,0){\framebox(26,18)[t]{\begin{picture}(24,12)(0,5)
\put(0,4){\line(1,0){24}}
\multiput(3,3)(4,0)5{$\bullet$}
\put(4,1){$\scriptstyle 5$}
\put(8,1){$\scriptstyle 4$}
\put(12,1){$\scriptstyle 3$}
\put(16,1){$\scriptstyle 2$}
\put(20,1){$\scriptstyle 1$}
\put(8,4){\oval(8,8)[t]}
\put(10,4){\oval(4,4)[t]}
\put(16,4){\line(0,1){10}}
\put(20,4){\line(0,1){10}}
\end{picture}}} 
\end{picture}
$$
\centerline{{\it Example 1:} Hasse diagram of the poset $\CK(\Gamma)$}
\end{figure}

%\begin{green}
By \cite[II.(4.1)]{macd}, we have ${\rm deg}\,h_\Gamma=h_{\alpha,\gamma}^\beta=8$;
with ${\rm deg}\,a_\alpha=20$ it follows that ${\rm deg}\, g_\Gamma=28$.
We use Theorem~\ref{thm-second-main} to compute
${\rm deg}\,g_{\Pi_i}={\rm deg}\,g_\Gamma-x(\Pi_i)$, for $i=1,\ldots,7$. Note that
\begin{itemize}
 \item the Klein tableau corresponding to $\Pi_1$ is dominant, 
   so $x(\Pi_1)=0$ and $\dim\,G_{\Pi_1}=28$;
 \item $x(\Pi_2)=1$ and $\dim\,G_{\Pi_2}=27$;
 \item $x(\Pi_3)=x(\Pi_4)=2$ and $\dim\,G_{\Pi_3}=\dim\, G_{\Pi_4}=26$;
 \item $x(\Pi_5)=3$ and $\dim\,G_{\Pi_5}=25$;
 \item $x(\Pi_6)=4$ and $\dim\,G_{\Pi_6}=24$;
 \item $x(\Pi_7)=5$ and $\dim\,G_{\Pi_7}=23$  
\end{itemize}
%\end{green}
 (compare Proposition \ref{prop:dim-zero-char}).

\smallskip
{\it Example 2:}
Consider the following triple of partitions $(\alpha,\beta,\gamma)$,
where $\alpha=(2,1,1)$, $\gamma=(3,2,1)$ and $\beta=(4,3,2,1)$. There are three LR-tableaux
of type $(\alpha,\beta,\gamma)$:

\vspace{-4mm}
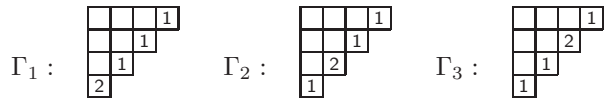
$$\Gamma_1:\quad
\setcounter{boxsize}{3}
\begin{picture}(18,12)(0,6)
\multiput(0,12)(3,0)4{\smbox}
\put(9,12){\numbox{1}}
\multiput(0,9)(3,0)3{\smbox}
\put(6,9){\numbox{1}}
\multiput(0,6)(3,0)2{\smbox}
\put(3,6){\numbox{1}}
\put(0,3){\smbox}
\put(0,3){\numbox{2}}
%\put(0,0){\numbox{2}}
\end{picture}
\Gamma_2:\quad
\setcounter{boxsize}{3}
\begin{picture}(18,12)(0,6)
\multiput(0,12)(3,0)4{\smbox}
\put(9,12){\numbox{1}}
\multiput(0,9)(3,0)3{\smbox}
\put(6,9){\numbox{1}}
\multiput(0,6)(3,0)2{\smbox}
\put(3,6){\numbox{2}}
\put(0,3){\smbox}
\put(0,3){\numbox{1}}
%\put(0,0){\numbox{2}}
\end{picture}
\Gamma_3:\quad
\setcounter{boxsize}{3}
\begin{picture}(18,12)(0,6)
\multiput(0,12)(3,0)4{\smbox}
\put(9,12){\numbox{1}}
\multiput(0,9)(3,0)3{\smbox}
\put(6,9){\numbox{2}}
\multiput(0,6)(3,0)2{\smbox}
\put(3,6){\numbox{1}}
\put(0,3){\smbox}
\put(0,3){\numbox{1}}
%\put(0,0){\numbox{2}}
\end{picture}
$$\vspace{0.5cm}

Note that $\Gamma_1\partleq\Gamma_2\partleq\Gamma_3$ and 
$$ \CK(\Gamma_1)=(\Pi_1\to\Pi_4\to\Pi_6)\; ,\; \CK(\Gamma_2)=(\Pi_2\to\Pi_5)\; ,\;\; \CK(\Gamma_3)=(\Pi_3),$$
where $\Pi_1,\ldots,\Pi_6$ are 
as indicated in the Hasse diagram.

\begin{figure}[tbh]
$$
\begin{picture}(100,100)

\put(30,88){$\Pi_6:$}

\put(40,80){\framebox(26,18)[t]{\begin{picture}(20,12)(0,5)
\put(0,4){\line(1,0){20}}
\multiput(3,3)(4,0)4{$\bullet$}
%\put(4,1){$\scriptstyle 5$}
\put(4,1){$\scriptstyle 4$}
\put(8,1){$\scriptstyle 3$}
\put(12,1){$\scriptstyle 2$}
\put(16,1){$\scriptstyle 1$}
%\put(10,4){\oval(12,12)[t]}
\put(10,4){\oval(12,12)[t]}
\put(8,4){\line(0,1){10}}
\put(12,4){\line(0,1){10}}
\end{picture}}}

\put(26,60){\vector(1,1){15}}
\put(78,60){\vector(-1,1){15}}
\put(-10,48){$\Pi_4:$}

\put(0,40){\framebox(26,18)[t]{\begin{picture}(20,12)(0,5)
\put(0,4){\line(1,0){20}}
\multiput(3,3)(4,0)4{$\bullet$}
%\put(4,1){$\scriptstyle 5$}
\put(4,1){$\scriptstyle 4$}
\put(8,1){$\scriptstyle 3$}
\put(12,1){$\scriptstyle 2$}
\put(16,1){$\scriptstyle 1$}
%\put(10,4){\oval(12,12)[t]}
\put(8,4){\oval(8,8)[t]}
\put(8,4){\line(0,1){10}}
\put(16,4){\line(0,1){10}}
\end{picture}}}

\put(70,48){$\Pi_5:$}
 
\put(80,40){\framebox(26,18)[t]{\begin{picture}(20,12)(0,5)
\put(0,4){\line(1,0){20}}
\multiput(3,3)(4,0)4{$\bullet$}
%\put(4,1){$\scriptstyle 5$}
\put(4,1){$\scriptstyle 4$}
\put(8,1){$\scriptstyle 3$}
\put(12,1){$\scriptstyle 2$}
\put(16,1){$\scriptstyle 1$}
%\put(10,4){\oval(12,12)[t]}
\put(12,4){\oval(8,8)[t]}
\put(12,4){\line(0,1){10}}
\put(4,4){\line(0,1){10}}
\end{picture}}}

%\put(26,100){\vector(3,1){50}}
%\put(78,100){\vector(-3,1){50}}
\put(13,20){\vector(0,1){15}}
\put(93,20){\vector(0,1){15}}
\put(-10,8){$\Pi_1:$}

\put(0,0){\framebox(26,18)[t]{\begin{picture}(20,12)(0,5)
\put(0,4){\line(1,0){20}}
\multiput(3,3)(4,0)4{$\bullet$}
%\put(4,1){$\scriptstyle 5$}
\put(4,1){$\scriptstyle 4$}
\put(8,1){$\scriptstyle 3$}
\put(12,1){$\scriptstyle 2$}
\put(16,1){$\scriptstyle 1$}
%\put(10,4){\oval(12,12)[t]}
\put(6,4){\oval(4,4)[t]}
\put(12,4){\line(0,1){10}}
\put(16,4){\line(0,1){10}}
\end{picture}}}

\put(70,8){$\Pi_3:$}

\put(80,0){\framebox(26,18)[t]{\begin{picture}(20,12)(0,5)
\put(0,4){\line(1,0){20}}
\multiput(3,3)(4,0)4{$\bullet$}
%\put(4,1){$\scriptstyle 5$}
\put(4,1){$\scriptstyle 4$}
\put(8,1){$\scriptstyle 3$}
\put(12,1){$\scriptstyle 2$}
\put(16,1){$\scriptstyle 1$}
%\put(10,4){\oval(12,12)[t]}
\put(14,4){\oval(4,4)[t]}
\put(4,4){\line(0,1){10}}
\put(8,4){\line(0,1){10}}
\end{picture}}}

\put(63,20){\vector(1,1){15}}
\put(43,20){\vector(-1,1){15}}
\put(30,8){$\Pi_2:$}

\put(40,0){\framebox(26,18)[t]{\begin{picture}(20,12)(0,5)
\put(0,4){\line(1,0){20}}
\multiput(3,3)(4,0)4{$\bullet$}
%\put(4,1){$\scriptstyle 5$}
\put(4,1){$\scriptstyle 4$}
\put(8,1){$\scriptstyle 3$}
\put(12,1){$\scriptstyle 2$}
\put(16,1){$\scriptstyle 1$}
%\put(10,4){\oval(12,12)[t]}
\put(10,4){\oval(4,4)[t]}
\put(4,4){\line(0,1){10}}
\put(16,4){\line(0,1){10}}
\end{picture}}}
\end{picture}
$$
\centerline{{\it Example 2:} Hasse diagram of the poset $\CK(\alpha,\beta,\gamma)$}
\end{figure}

\bigskip
\textbf{Acknowledgement.} The first named author wishes to thank 
Florida Atlantic University, where the authors started to write this paper, for the invitation.
The second named author gratefully recognizes invitation and support for his visit to 
Nicolaus Copernicus University in Toru\'n during December 2011. 
Both authors want to thank Grzegorz Zwara for the discussion on the 
geometric part of this paper.

\medskip
\textbf{Dedication.}  
The authors wish to dedicate this paper to Professor Daniel Simson on the occasion of his 70th birthday.
Throughout our careers, we have known Prof.\ Simson as a devoted scientist who has nurtured 
our interests towards attractive areas in algebra.  We thank him for taking part in our
mathematical development and for supporting us in a variety of ways.

Address of the authors:

\parbox[t]{5.5cm}{\footnotesize\begin{center}
              Faculty of Mathematics\\
              and Computer Science\\
              Nicolaus Copernicus University\\
              ul.\ Chopina 12/18\\
              87-100 Toru\'n, Poland\end{center}}
\parbox[t]{5.5cm}{\footnotesize\begin{center}
              Department of\\
              Mathematical Sciences\\ 
              Florida Atlantic University\\
              777 Glades Road\\
              Boca Raton, Florida 33431\end{center}}

\smallskip \parbox[t]{5.5cm}{\centerline{\footnotesize\tt justus@mat.umk.pl}}
           \parbox[t]{5.5cm}{\centerline{\footnotesize\tt markus@math.fau.edu}}

\end{document}